\mag=\magstep1

\documentclass[10pt,a4paper,dviout]{amsart} 
\usepackage{amssymb,latexsym} 
\usepackage{color} 
\usepackage{graphicx}
\usepackage{comment}
\usepackage{datetime}
\usepackage{comment}
\usepackage{mathtools}

\usepackage{amsmath, amssymb, eucal, amscd, color, amsthm}

\def\R{\Bbb R}
\def\si{\sigma}

\def\cal{\mathcal}
\def\co{{\cal O}}
\def\Q{\Bbb Q}
\def\part{\partial}

\def\we{\wedge}
\def\e{\epsilon}
\def\dis{\displaystyle}

\def\P{\mathbb P}
\def\C{\mathbb C}
\def\s{{\square^n}}

\def\sd{\operatorname{sd}}

\def\codim{{{\rm codim}\,}}
\def\ov{\overline}

\def\p1{\prec}

\def\<{\langle}
\def\>{\rangle}

\def\sd{\operatorname{sd}}

\newcommand{\sign}{\operatorname{sign}}
\newcommand{\Alt}{\operatorname{Alt}}

\newtheorem{theorem}{Theorem}[section] 
\newtheorem{proposition}[theorem]{Proposition} 
\newtheorem{corollary}[theorem]{Corollary} 
\newtheorem{definition}[theorem]{Definition} 
 
\newtheorem{remark}[theorem]{Remark} 
\newtheorem{lemma}[theorem]{Lemma} 
\newtheorem{lemma-definition}[theorem]{Lemma-Definition} 
\newtheorem{proposition-definition}[theorem]{Proposition-Definition}

\newtheorem{notation}[theorem]{Notation}

\newcommand{\CC}{{\mathbb{C}}}
\newcommand{\PP}{{\mathbb{P}}}
\newcommand{\QQ}{{\mathbb{Q}}}
\newcommand{\RR}{{\mathbb{R}}}
\newcommand{\ZZ}{{\mathbb{Z}}}

\newcommand{\Ker}{\operatorname{Ker}}

\newcommand{\al}{\alpha}  \newcommand{\ga}{\gamma} 
 \newcommand{\eps}{\epsilon} \newcommand{\la}{\lambda}

\newcommand{\mapright}[1]{%
  \smash{\mathop{%
    \hbox to 1cm{\rightarrowfill}}\limits^{#1} } } 

\newcommand{\smapr}[1]{%
  \smash{\mathop{%
    \hbox to 0.5cm{\rightarrowfill}}\limits^{#1} } } 
\newcommand{\maprb}[1]{%
  \smash{\mathop{%
    \hbox to 1cm{\rightarrowfill}}\limits_{#1} } } 
\newcommand{\mapleft}[1]{%
  \smash{\mathop{%
    \hbox to 1cm{\leftarrowfill}}\limits^{#1} } }

\newcommand{\maplb}[1]{%
  \smash{\mathop{%
    \hbox to 1cm{\leftarrowfill}}\limits_{#1} } }

\def\alt{\text{alt}}

\def\AC{{\cal A\cal C}}
\def\res{\operatorname{Res}}

\def\sd{\text{sd}}

\makeatletter
  
  \@addtoreset{equation}{section}
 \makeatother

\makeindex
\usepackage{version}	
\begin{document}

\title{Semi-algebraic chains on projective varieties and the Abel-Jacobi map for higher Chow cycles
} 
\author{Kenichiro Kimura}

\maketitle

\setcounter{tocdepth}{3}

\markboth
{}
{}

\thispagestyle{empty}
\setcounter{tocdepth}{1}
\section{Introduction}
The main object of study in this paper is the cohomology groups of smooth quasi-projective 
complex varieties. The reader may be skeptical about finding anything new in general about this
subject. What we are going to do is to describe the cohomology groups of a smooth quasi-projective
variety $U$ relative to a normal crossing divisor $\bf H$, in terms of  
$\delta$-{\it admissible chains.} 
Roughly speaking, a $\delta$-admissible chain on $U$ relative to a subvariety $Y$ of $U$ is a simplicial
semi-algebraic chain $\ga$ such that the support of $\ga$ and that of $\delta\ga$ (the boundary of $\ga$) meet $Y$
properly.   One of the merits of $\delta$-admissible chains is that 
they admit pull-back to $Y$. By this property, for a normal
crossing divisor $\bf H$ on $U$ we construct a certain complex $AC^*(U,\bf H)$ of $\delta$-admissible chains
such that we have an isomorphism
\[H^j(AC^*(U,{\bf H}))\simeq H^j(U,{\bf H};\ZZ).\]
The proof is not quite elementary, and we needed to use some sheaf theory.

\noindent We can describe the duality pairing between the de Rham and the singular cohomology via
integral on $\delta$-admissible chains. Let $\varphi$ be a smooth $p$-form on $U$ with compact support which has logarithmic
singularity along $\bf H$, and let $\ga$ be a $\delta$-admissible $(p+1)$-chain on $U$. Then in Proposition
\ref{C-S} we show that  the equality 
\[\int_{\part_{\bf H}\ga} \res \varphi=\int_{\delta\ga}\varphi-\int_\ga d\varphi\] 
holds. For the definition of the differential $\part_{\bf H}$ see Definition \ref{doubleAC}.  This
formula is a generalization of the Stokes formula, and we call this equality the Cauchy-Stokes formula. Note that the integral of a differential form with log poles 
on a $C^\infty$-
chain does not necessarily converge, even if the chain meets the faces properly.
See the example in Remark \ref{counter example} (1).
By the Cauchy-Stokes formula we have a certain pairing 
\[AC^*(U,{\bf H})\otimes A^*_c(U)(\log {\bf H})\to \C[-2\dim U]\]
which is a map of complexes, and this induces the duality pairing
\[H^{2\dim U-j}(U,{\bf H};\QQ)\otimes H^j(X-{\bf H},D;\C)\to \C.\]
Here $X$ is a smooth projective complex variety,  $D$ is a closed subset of $X$
such that $X-D$ is isomorphic to $U$.
As an application, in \S 4 we show that the Abel-Jacobi map for higher Chow cycles can be described 
in terms of $\delta$-admissible chains, 
which can be regarded as a natural generalization of 
the original definition by Griffiths of the Abel-Jacobi map for ordinary algebraic cycles. 
Details are in \S 4, but we point out one advantage of the admissible chains. Let $Y$
be a smooth projective complex variety, and 
$Z\in z^p(Y,n)^{\alt}$ be a higher Chow cycle such that $\part_\square Z=0$. See \S 4 for the definitions. As is described in
\cite{B3} and \cite{Sch}, the definition 
of the Abel-Jacobi map starts with defining a cohomology class of $Z$ in $H^{2p}_{|Z|}(\s\times Y,\part \s\times Y;\QQ(p))$, which is the cohomology
of $\s\times Y$ with support on $Z$, relative to the faces of $\square^n\times Y$. By using the admissible chains we can construct a certain 
complex which computes this cohomology group, such that the class of the cycle $Z$ is represented by $Z$ itself. If $Z$ is homologous to zero, then there is a chain $\Gamma$ which has $Z$ as the boundary. It is an immediate consequence
of the construction that the class of $\Gamma$ gives the Abel-Jacobi image of $Z$. This construction works also for open
varieties, and for  relative higher Chow cycles. As an example, we will see that the Hodge realization of the cycles of polylogarithms
constructed in \cite{BK} can be described in terms of the Abel-Jacobi map of certain open varieties.  
 
  The problem of describing the Abel-Jacobi maps for higher Chow cycles has been considered by other authors too. This paper
is partly inspired by \cite{KLM} and \cite{KL}. Our description of the Abel-Jacobi map can be regarded as a generalization 
of the {\it geometric interpretation} in \S 5.8 of \cite{KLM} to quasi-projective varieties.  Even in the case of projective varieties, 
our description is somewhat different from the one given in \cite{KLM}.

\section{The complex of admissible chains}

Let $k$ be a non-negative integer. A $k$-simplex in an Euclidean 
space $\R^n$ is the convex hull of 
affinely independent points $a_0,\cdots,
a_k$  in $\R^n$.
{\it A finite simplicial complex} of $\R^n$ is a finite set $K$ consisting of  simplexes 
 such that (1) for all $s\in K$, all the faces of $s$
belong to $K$,  (2) for all $s,t\in K$, $s\cap t$ is either the empty set
or a common face of $s$ and $t$.
We denote by $K_p$ the set of $p$-simplexes
of $K$.
For a finite simplicial complex $K$, the union of  simplexes
in $K$ as a subset of $\R^n$ is denoted by $|K|$.

As for the definition of semi-algebraic set and their fundamental
properties, see \cite{BCR}.
\begin{theorem}[\cite{BCR}, Theorem 9.2.1]
\label{thm:semi-algebraic triangulation}
Let $P$ be a compact semi-algebraic subset of $\R^m$. The set $P$ is triangulable, i.e. there exists
a finite simplicial complex $K$ and a semi-algebraic homeomorphism
$\Phi_K:\,|K|\to P$. Moreover, for a given finite family 
$S=\{S_j\}_{j=1,\cdots,q}$ of semi-algebraic
subsets of $P$, we can choose a finite simplicial complex $K$
and a semi-algebraic homeomorphism
$\Phi_K:\,|K|\to P$ such that every $S_j$ is the
union of a subset of $\{\Phi_K(\sigma^\circ)\}_{\sigma\in K}.$
Here $\si^\circ$ is the interior of $\si$, which is the complement of the union of all proper faces
of $\si$.
\end{theorem}
\begin{remark}
\label{facewise regular embedding}
\begin{enumerate}

\item By  \cite{BCR} Remark 9.2.3 (a), the map $\Phi_K$ can be taken so that the map 
$\Phi_K$ is {\it facewise regular embedding} i.e. for each $\sigma\in K$, 
$\Phi_K(\sigma^\circ)$  is a  regular
submanifold of $\R^m$. 

\item The pair $(K, \Phi_K)$ 
as in Theorem \ref{thm:semi-algebraic triangulation}
is called a 
{\it semi-algebraic triangulation}\index{semi-algebraic triangulation} of $P$; we will then identify $|K|$ with $P$. 
A projective real or complex variety $V$ is regarded  as a compact semi-algebraic subset 
of an Euclidean space by \cite{BCR} Theorem 3.4.4, thus the above theorem applies to $V$. 
\end{enumerate}
\end{remark}
We recall some terminology of piecewise-linear topology. Let $K$ be a simplicial complex
and $L$ be a subcomplex of $K$. 
$L$ is a {\it full subcomplex} of $K$, 
if all the vertices of a simplex $\si$ in $K$ belong to $L$, then $\si$ belongs to $L$.
a {\it derived subdivision of $K$ {\rm mod} $L$} is obtained by
starring each simplex of $K$ not contained in $L$. See for example \cite{RS} page 20 for
more detail. If we star each simplex not contained in $L$ at its barycenter, we obtain the {\it barycentric subdivision of
$K$ {\rm mod} $L$}, which is denoted by $\text{sd}K$ mod $L$. The {\it simplicial neighborhood of $L$ in $K$}, denoted by
$N(L,K)$, is defined to be $\{\si\in K|\,\exists \eta,\,\si\text{ is a face of } \eta \text{ and } \eta\cap L
\neq \emptyset\}$. Also the {\it simplicial complement of $L$ in $K$}, denoted by  $C(L,K)$, is defined to be $\{\si\in K|\, \si\cap L=\emptyset\}.$ The intersection $N(L,K)\cap
C(L,K)$ is denoted by $\dot{N}(L,K)$. Suppose that 
$L$ is a full subcomplex of $K$, and let $K'$ be a derived subdivision of $K$ mod $L\cup 
C(L,K)$. The polytope $|N(L,K')|$ is said to be a {\it regular neighborhood} of $|L|$ in $|K|$.

\begin{notation}
\label{loose notation}
Let  $K$ be a simplicial complex. 
For a subcomplex $L$ of $K$,  the 
space $|L|$ is a subspace of $|K|$.  
A subset of $|K|$ of the form
$|L|$ is also called a subcomplex. If a subset $S$ of $|K|$ is equal to $|M|$ for a subcomplex $M$
of $K$, then $M$ is often denoted by $K\cap S$.
For two simplexes $\si$ and $\eta$ we denote  $\si \prec \eta$ if $\si$ is a face of $\eta$. 
For a complex $K$ and its subcomplex $L$, by writing $L\lhd K$ we mean that $L$
is a full subcomplex of $K$.  

\end{notation}

Let $X$ be a smooth projective variety over $\CC$ of dimension $m$, $D$ be a closed subset of $X$ and 
${\bf H}=H_0\cup\cdots \cup H_t$ be a strict 
normal crossing divisor on $X$. We write $U=X-D$. The {\it faces} of $X$ are intersections of several $H_j$'s. 
For a subset $I$ of $T:=\{0,\cdots, t\}$, we write the face $\underset{i\in I}\cap H_i$
by $H_I$.
In the following we suppose that each triangulation $K$ of  $X$ is semi-algebraic , and that the subset
$ D$ is a subcomplex of $K$.  
Let $K$ be a  triangulation of $X$. We denote by $C_\bullet(K;\ZZ)$ resp. $C_\bullet(K,D;\ZZ)$ 
the chain complex resp. the relative chain 
complex of $K$. An element of $C_p(K;\ZZ)$ for $p\geq 0$ is written as $\sum a_\si \si$
where the sum is taken over $p$-simplexes of $K$. By doing so, it is agreed upon that an orientation
has been chosen for each $\si$.  By abuse of notation, an element of $C_\bullet(K,D;\ZZ)$ 
is often described similarly. 

\begin{definition}
 For an element $\gamma=\sum a_\si \sigma$ of
$C_{p}(K;\ZZ)$, we define
the support\index{support} $|\gamma|$ of $\gamma$ as the subset of $|K|$
 given by
\begin{equation}
\label{support}
|\gamma|=\bigcup_{
\substack{\sigma\in K_p\\ a_\sigma\neq 0}} \sigma.
\end{equation}
\end{definition}
For
an element $\gamma=\sum a_\si \sigma$ of
$C_{p}(K;\ZZ)$, 
$|\gamma|$ is sometimes regarded as a subcomplex of $K$.

\begin{definition}
\label{def:semi-alg current}
Let $p \geq 0$ be an integer.
\begin{enumerate}

\item(Admissibility) A semi-algebraic subset $S$ of $X$
is 
said to be admissible\index{admissible} if for each face $H$, 
the inequality 
$$
\dim(S\cap (H-D))\leq \dim S -2\,\codim H
$$
holds.  Here note that $\dim(S\cap (H- D))$ and $\dim S$ are the dimensions as semi-algebraic sets,
and $\codim H$ means the codimension of the subvariety $H$ of $X$.

\item Let $\gamma$ be an element of $C_p(K,D; \ZZ)$.
Then $\ga$ is 
said to be admissible if 
the support of a representative of $\ga$ in $C_p(K;\ZZ)$ is admissible. This condition is independent of the choice
of a representative. 
\item We set 
$$
AC_{p}(K,D;\ZZ)
=\{\gamma\in C_{p}(K,D;\ZZ))
\mid 
\gamma \text{ and } \delta \gamma \text{ are admissible }\}.
$$
\index{$AC_{p}(K,D;\Zz^{q})$}
We call an element of $AC_{p}(K,D;\ZZ)$ a $\delta$-admissible chain.
\end{enumerate}
\end{definition}

\subsection{Subdivision and inductive limit}

\begin{definition}
Let $(K,\,\Phi_K:\,|K|\to P)$ be a  triangulation of  a compact semi-algebraic set $P$. Another 
triangulation $(K',\,\Phi_{K'}:\,|K'|\to P)$ is {\rm a subdivision of}\index{subdivision} $K$ if :
\begin{enumerate}
\item The image of each simplex
of $K'$ under the map $\Phi_{K'}$ is contained in  the image of a simplex of $K$
under the map $\Phi_K$.
\item The image of each simplex of $K$ under the map $\Phi_K$ is the union of the images of
simplexes of $K'$ under $\Phi_{K'}$.
\end{enumerate}
\end{definition}
If $K'$ is a  subdivision of a  triangulation $K$, there is 
a natural homomorphism of complexes
$\lambda:
C_{\bullet}(K,D;\ZZ)\to 
C_{\bullet}(K',D;\ZZ)$ called the subdivision operator. See for example  
 \cite{Mu} Theorem 17.3 for the definition. For a simplex $\si$ of $K$, the chain $\la(\si)$ is carried
by $K'\cap \si$, and so that the map $\la$ sends $AC_{\bullet}(K,D;\ZZ)$ to 
$AC_{\bullet}(K',D;\ZZ)$. 
By Theorem \ref{thm:semi-algebraic triangulation} two  semi-algebraic triangulations
have a common subdivision. 
Since the map $\lambda$
and the differential $\delta$ commute, the complexes 
$C_{\bullet}(K,D;\ZZ)$ and 
$AC_{\bullet}(K,D;\ZZ)$ form
inductive systems indexed by triangulations $K$ of $X$.
\begin{definition}
\label{def:definition of AC with inductive limit}
We set
\begin{align*}
C_{\bullet}(X,D;\ZZ) =
\underset{\underset{K}\longrightarrow}{\lim } \ 
C_{\bullet}(K,D;\ZZ), \quad
AC_{\bullet}(X,D;\ZZ) =
\underset{\underset{K}\longrightarrow}{\lim } \ 
AC_{\bullet}(K,D;\ZZ).
\end{align*}
Here the limit is taken on the directed set of  triangulations.

\end{definition}
\index{$C_{\bullet}(X,\bold D;\Zz^{\bullet})$}
\index{$AC_{\bullet}(X,\bold D;\Zz^{\bullet})$}


A proof of the following Proposition is given in \cite{part II} Appendix A.
\begin{proposition}[Moving lemma] 
\label{prop: moving lemma}
The inclusion of complexes
\begin{equation}
\label{moving quasi-iso}
\iota:\,\,AC_{\bullet}(X, D; \ZZ) 
\to C_{\bullet}(X,D; \ZZ)
\end{equation}
is a quasi-isomorphism.
\end{proposition}

\begin{definition}[Good triangulation]
\label{good triangulation}
We define a family $\cal L$ of subsets of $X$ by
$$
\cal L=\{H_{I_1}\cup \cdots\cup H_{I_k}\}_{(I_1, \dots, I_k)},
$$ 
where $H_{I_j}$ are  faces of $X$. In short, a member of $\cal L$ is the union of several
 faces.
A finite semi-algebraic triangulation $K$ of $X$ is called a good triangulation \index{good triangulation} if $K$
satisfies the following conditions.
\begin{enumerate}
\item The divisor $D$ is a subcomplex of $K$.
\item The map $\Phi_K:\,|K|\to X$ is facewise regular embedding. cf. Remark \ref{facewise regular embedding}.
\item 
\label{faces are full subcomplex}
Each element $L_i\in \cal L$ is a full subcomplex of $K$ i.e.  there exists a full subcomplex $M_i$ of $K$ such that $L_i=|M_i|$. 
\end{enumerate}

\end{definition}
In particular, if $K$ is a good triangulation,
then for any simplex $\si$ of 
$K$ and $L_i\in \cal L$,
the intersection $\si\cap  L_i$ is a (simplicial) face of $\si$. This is the primary reason to consider the condition (3).

\begin{remark} 
\label{barycentric subdiv is good}
If $L$ is a subcomplex
of $K$, then ${\rm sd} L$ is a full subcomplex of ${\rm sd} K$ (See for example Exercise 3.2 of \cite{RS}).   It follows that 
if  $K$ is a semi-algebraic triangulation of 
$X$ which  is a facewise regular embedding, and such that
$D$ and each $L_i\in \cal L$ are  subcomplexes of $K$, then ${\rm sd} K$ is a good triangulation.
\end{remark}

In the following, each triangulation $K$ of $X$ is assumed to be good in the sense of Definition
\ref{good triangulation}.

\subsection{The cap product with a Thom cocycle}
\label{subsec:cap prod and thom cocycle}
\subsubsection{Simplicial cap product}
\begin{definition}[Ordering of complex, 
good ordering]
\label{def:good simplical ordering}
Let  $K$ be a good triangulation of $X$. 
\begin{enumerate}
\item
A partial ordering on the set of vertices in $K$ is called 
an ordering\index{ordering} of $K$,
if the
     restriction of the ordering to each simplex is a total ordering.
\item  
Let $L$ be a subcomplex of $K$.
An ordering of $K$ is said to be good with respect to $L$ if it satisfies the following condition.
If  a vertex $v$ is on $L$  and  $w\geq v$ for a vertex $w$, then  $w \in L$. 
\index{good ordering}
\end{enumerate}
\end{definition}

We denote by $[a_0,\cdots, a_k]$\index{$[a_0,\cdots, a_k]$} the simplex spanned by $a_0,\cdots, a_k$.
Let $\co$ be a good ordering of $K$ with respect to a subcomplex $L$.
We recall the definition of the cap product\index{cap product}
\newline $\overset{\co}\cap:C^p(K)\otimes C_k(K) \to C_{k-p}(K)$. 
For a simplex $\al=[v_0,\cdots,v_k]$ such that $v_0< \cdots <v_k$ and $u\in C^p(K)$, we define 
\begin{equation}
\label{def:simplicial cap product}
u\overset{\co}\cap \alpha=u([v_0, \dots, v_p])[v_{p},\dots, v_k].
\end{equation}
 One has the boundary formula
\begin{equation}
\label{boundary formula}
\delta(u\overset{\co}\cap \alpha)=(-1)^p(u\overset{\co}\cap (\delta\alpha)
-(du)\overset{\co}\cap \alpha)
\end{equation}
where $du$ denotes the coboundary of $u$, see \cite{Hat}, p.239 (note the difference
in sign convention from \cite{Mu}). 
Thus if $u$ is a cocycle, $\delta(u\overset{\co}\cap \alpha)
=(-1)^pu\overset{\co}\cap (\delta\alpha).$

\subsubsection{Cohomology class of a subvariety}

For a subvariety $V$ of $X$ of codimension $p$,  there exists a cohomology class $cl(V)\in H^{2p}_V(X, \ZZ)$.
In the case of simplicial cohomology, it is described as follows. Let $K$ be a triangulation of $X$ such that 
there exists a full subcomplex $L$ of $K$ with $V=|L|$. Let $K'$ be a derived subdivision of $K$ mod
$L\cup C(L,K)$. Set $N=N(L,K')$ and $C=C(L,K')$. In this situation the cohomology group $H^{2p}_V(X,\ZZ)$
is equal to $H^{2p}(C^*(K,C;\ZZ))$, and by Lefschetz duality Theorem 3.43 \cite{Hat}  we have an isomorphism
\[\cap\eta_X:\,H^{2p}(C^*(K,C;\ZZ))\to H_{2(m-p)}(C_*(N;\ZZ)).\]
Here $\eta_X$ is the fundamental cycle of $X$. The element $cl(V)$ is the one such that
$cl(V)\cap \eta_X $ equals the homology class of the cycle $V$. A cocycle in $C^{2p}(K,C;\ZZ)$ which
represents $cl(V)$ is called a Thom cocycle of $V$, and is denoted by $T_V$.

\begin{proposition}
\label{basic property of face map 1}
Let $V$ be a subvariety of $X$ of codimension $p$, $K$ a triangulation of $X$ for which $V=|L|$ for a full subcomplex $L$.
Let $K'$ be a derived subdivision of $K$ mod $L\cup C(L,K)$,  $T_V\in C^{2p}(K, C(L,K');\ZZ)$ be a Thom cocycle of $V$  and $\co$ a good ordering of $K'$ with respect to $L$.
\begin{enumerate}

\item The map $T_V\overset{\co}{\cap}$ 
and the topological differential $\delta$ commute. 
\item 
The image of the homomorphism 
$T_V\overset{\co}{\cap} $ is contained in $C_{k-2p}(L;\ZZ)$,
 As a consequence, we have a homomorphism of
 complexes
\begin{equation}
\label{cap prod no cond}
T_V\overset{\co}{\cap} :\,\,C_k(K', D; 
\ZZ) \to
C_{k-2p}(L, L\cap D; \ZZ).
\end{equation}
\end{enumerate}
\end{proposition}
\begin{proof}
(1). Since $T_V$ is a cocycle of even degree, we have
$\delta(T_V\overset{\co}\cap \sigma)=T_V\overset{\co}\cap (\delta\sigma)$ for a simplex $\si$.

\noindent (2). Let $\si=[v_0,\cdots, v_n]\in C_n(K')$ with $v_0<\cdots <v_n.$ If $v_{2p} \notin   V$, then $[v_0,\cdots,v_{2p}] \cap  V=\emptyset$ and
 $T_V([v_0,\cdots,v_{2p}])=0$ since the cochain $T$ vanishes on $C(L,K')$.
If $v_{2p} \in L$, then the vertices $v_{2p},\cdots, v_{n}$ are on $  L$, and we have $[v_{2p},\dots, v_n]\subset  L$
since $  L\lhd K$.
Thus the assertion holds.
\end{proof}

\subsubsection{Independence of $T_V$ and  ordering}
\label{independece of c}

\begin{proposition}
\label{prop face map first properties}
Let $K$ be a good triangulation of $X$ and $H$ be a codimension $p$ face of $X$. Set $L=K\cap H$. 
Let $T_H\in C^{2p}(K', C(L,K');\ZZ)$ be a Thom cocycle of $H$, $\co$ a good ordering of $K'$ with respect to $L$,
and  $\gamma$ be an element of $AC_{k}(K',D;\ZZ)$.
Then we have the following. 
\begin{enumerate}
\item
The chain  $T_H\overset{\co}{\cap} \gamma$
is an element in $AC_{k-2p}(L,L\cap D; \ZZ)$.
\item
\label{indep of T for face map}
The chain $T_H\overset{\co}{\cap} \gamma$ is
independent of the choice of a 
Thom cocycle $T_H$ and a good ordering $\co$. 
Thus the map 
$$
T_H\overset{\co}\cap:
AC_k(K, D; 
\ZZ) \to
AC_{k-2p}(L, L\cap D; \ZZ).
$$
induced by (\ref{cap prod no cond})
 is denoted by $T_H\cap$.
\item
\label{comatibility for subdivision}
Let $M$ be a good subdivision of $K$.
Let $M'$ be a derived subdivision of $M$ mod $(M\cap H)\cup C(M\cap H, M)$ and $T'_H
\in C^{(2p)}(M', C(M\cap H, M');\ZZ)$ be a Thom cocycle of $H$. Then we have the following 
commutative diagram.
\begin{equation}
\label{commutative diagram for subdivision}
\begin{matrix}
AC_k(K,D;\ZZ)
&\xrightarrow{T_H\cap}&
    AC_{k-2p}(L,L\cap D; \ZZ) 
\\
\lambda \downarrow & & \downarrow\lambda
\\
AC_k(M',D;\ZZ)
&\xrightarrow{T'_H\cap}&
    AC_{k-2p}(M'\cap H,M'\cap (D\cap H); \ZZ)
\end{matrix}
 \end{equation}
where the vertical maps $\lambda$ are subdivision operators.
\end{enumerate}
\end{proposition}
\begin{proof} (1). 
For an element $z\in C_k (K', D;\QQ)$, 
we have $T_H\overset{\co}{\cap} z\in  C_{k-2p} (L, D;\QQ)$
by Proposition \ref{basic property of face map 1} 
(2). By the definition of the cap product, we see that 
the set $|T_H\overset{\co}{\cap} z|\subset |z|\cap  H$. It follows that if  $z$ is admissible i.e. $|z|-D$ meets all
the faces properly, then  $|T_H\overset{\co}{\cap} z|- D $ meets all  the 
faces of $  H$ properly.  Similarly, if  the chain  $\delta z$ is admissible, then  $T_H\overset{\co}{\cap}(\delta  z)$
is admissible in $  H$.  By Proposition \ref{basic property of face map 1} (1) we have the equality $\delta(T_H\overset{\co}{\cap} z)=
T_H\overset{\co}{\cap}(\delta  z)$.

(\ref{indep of T for face map}) A proof is given in \cite{part II} Section B.2.

(3)  A proof is given in \cite{part II} Section B.3.

 \end{proof}
By taking the inductive limit 
of the homomorphism
$$
T_H\cap :AC_{\bullet}(K,D;\QQ)\to
   AC_{\bullet-2p}(L, D\cap L; \QQ).
$$
for subdivisions,
we get 
a homomorphism
\begin{equation}
 \label{face map simple sheaf}
T_H\cap:\,AC_\bullet(X, D;\QQ)\to 
AC_{\bullet-2p}(H, D\cap H; \QQ).
\end{equation}

\begin{definition}
\label{face map}
The map $T_H\cap$ of (\ref{face map simple sheaf}) is denoted by $\part_H$, and called the
face map of the face $  H$.
\end{definition}
\begin{proposition}
Let $H_1$ resp. $H_2$ be a face of codimension $p_1$ resp. $p_2$
 which meet properly with each other. Set $H_{12}=H_1\cap H_2$. We have the equality
\[\part_{H_1}\part_{H_2}=\part_{H_2}\part_{H_1},\,\, AC_j(K,D;\QQ)\to AC_{j-2(p_1+p_2)}(K\cap H_{12}, D\cap H_{12}; \QQ).\]
\end{proposition}

\begin{proof} A proof is given in \cite{part II} Section B.4.
\end{proof}

\begin{definition} 
\label{doubleAC}
Let $AC^{*,*}(U)$ be the double complex defined by
\[AC^{p,q}(U)=\underset{\sharp I=p}\oplus AC_{2m-2p-q}(H_I,D\cap H_I;\ZZ)\]
where the first differential  $\part_{\bf H}$ is defined by
\[(\part_{\bf H}\ga)_{(\al_0<\cdots<\al_p)}=\sum_{i=0}^p(-1)^i\part_{H_{\al_i}}(\ga_{(\al_0<\cdots \widehat{\al_i}
\cdots <\al_p)})\]
and the second differential is $\delta$ (the topological differential). For a finite set $I$, $\sharp I$ denotes the cardinality of $I$. 
The simple complex associated to $AC^{*,*}(U)$ is denoted by $AC^*(U,{\bf H};\ZZ)$.
\end{definition}

\begin{theorem}
\label{cohom isom}
For a non-negative integer $j$, we have an isomorphism
\[H^j(U,{\bf H};\ZZ)\simeq H^j(AC^*(U,{\bf H};\ZZ)).\]
Here cohomology on the  left hand side is the singular cohomology of $U$ relative to $\bf H$.
\end{theorem}

\begin{proof} 
\begin{definition}
Let $V$ be an open subset of $X$. For a semi-algebraic triangulation $K$ of $X$ , we denote by $C_*(K, (X-V))$ resp. $AC_*(K, (X-V))$
the quotient of $C_*(K)$ resp. $AC_*(K)$ by the chains 
contained in $(X-V)$.  More precisely,
\[\begin{array}{ll}
&AC_j(K,(X-V))\\
=&
\dis\frac{\{\gamma\in C_j(K)\,|\,|\ga|-D\text{ and }|\delta\ga|-D\text{ meet the faces 
properly}\}}{\{\text{ those }\ga\text{ such that }|\ga|\subset X-V\}.}
\end{array}\]
The associated sheaf to the presheaf $V\mapsto 
\underset{K}\varinjlim \,C_*(K, (X-V))$ resp. $V\mapsto 
\underset{K}\varinjlim \,AC_*(K, (X-V))$ is denoted by $\cal C_*$ resp. $\AC_*$. 
Here the limit is taken over semi-algebraic subdivisions of $X$. We define the associated cohomological
complex  by $\cal C^i=\cal C_{-i}$ resp. $\AC^i=\AC_{-i}$.
\end{definition}

\begin{proposition}
\label{resolution}
The complexes of sheaves $\cal C^*[-2m]$ and $\AC^*[-2m]$ are resolutions of the constant sheaf
$\ZZ$.

\end{proposition}

\begin{proof}
We give a proof of the case of $\AC^*[-2m]$. The proof for
$\cal C^*[-2m]$ is similar and simpler. We recall the following theorem.
\begin{theorem}{(\cite{Ze} Ch.6, Theorem 15}) 
\label{Zeeman moving}
Let $M$ be a compact $PL$-manifold.  Let $\cal X$,
$X_0$ and $H$ be subpolyhedra of $M$ such that $X_0\subset \cal X$ and ${\cal X}-X_0
\subset \overset\circ{M}$ (the interior of $M$). Then there exists
an ambient $PL$ isotopy $h:\,\,M\times [0,1]\to M$ which fixes $X_0$ and $\dot{M}$ (the boundary of $M$), and such that $h_1({\cal X}-X_0)$ is in general
position with respect to $H$ i.e. the inequality
\[\dim (h_1({\cal X}-X_0)\cap H)\leq \dim ({\cal X}-X_0)+\dim H-\dim M\]
holds. Here $h_t(m)=h(m,t)$ for $m\in M$ and $t\in [0,1]$.
\end{theorem}

The isotopy $h$ can be made arbitrarily small in the following sense. Given a positive number
$\e>0$, there exists a $PL$ isotopy $h$ as above such that for any point $(x,t)\in M\times [0,1]$,
the inequality $|h_0(x)-h_t(x)|<\e$ holds. Here $|\cdot|$ is the norm of the Euclidean space in which  $M$ is
contained. 

For a point $x\in X$ and its neighborhood $V_x$, let $\ga\in AC_j(K, (X-V_x))$ be a cycle i.e.
$|\delta\ga| \subset (X-V_x)$. 
There exists a subdivision $K'$ of $K$ of which  $x$ is a vertex of 
$K'$, and the simplicial neighborhood of $x$ in $K'$ is contained in $V_x$. We replace $V_x$ with
the interior of the simplicial neighborhood of $x$.
By excision, we have
$$H_j(X, (X-V_x);\ZZ)\simeq H_j(V_x, \delta V_x;\ZZ)
=\begin{cases}
0& j\neq 2m\\
\ZZ &j=2m.
\end{cases}
$$  It follows that if $j\neq 2m$, there exists $\Gamma\in C_{j+1}(K')$ such that 
$|\delta\Gamma-\ga|\subset  (X-V_x).$  
By applying Theorem \ref{Zeeman moving} to the case where
${\cal X}=|\Gamma|\cup |\ga|\cup D$ and $X_0=|\ga|\cup D$, 
we can move $\Gamma$ to a chain
$\Gamma'$ keeping $|\ga|\cup D$ fixed, so that $|\Gamma'|-D$ and 
$|\delta\Gamma'|-D$
meet the faces properly, and for a smaller neighborhood $V_x'$ of $x$, we have
$|\delta\Gamma'-\gamma| \subset (X-V_x')$. For $j=2m$, the fundamental 
class of $X$ gives the inclusion $\ZZ\to \AC_{2m}$, the image of which is isomorphic
to $H_{2m}(V_x, \delta V_x;\ZZ)$.

\end{proof}

\begin{proposition}
\label{fineness}
Let $j:\,U\to X$ be the inclusion map. The complex of sheaves $j^*\cal C^*$ is a complex of fine sheaves.
\end{proposition}
\begin{proof}
We recall the definition of a fine sheaf from \cite{Sw} page 74 Definition. A sheaf 
$\cal F$ 
on a space $X$ is fine
if, for every locally finite covering $\{U_\al\}$ of $X$, there exist endomorphisms $l_\al:\, \cal F\to \cal F$
such that
\begin{enumerate}
\item The support of $l_\al$ is contained in the closure of $U_\al$.
\item $\sum l_\al=\text{id}.$
\end{enumerate}
Let $\{U_\al\}$ be a locally finite covering of $U$. Let $(K,M)$ be a semi-algebraic 
triangulation of the pair $(X,D)$ such that $M\lhd K$.  Let
$\pi:\, K\to [0,1]$ be the simplicial map defined on each vertex $v$ by
\begin{equation}
\label{projection}
\pi(v)=
\begin{cases}
0& v\in M\\
1& v\notin M.
\end{cases}
\end{equation} 
Since $M\lhd K$, $\pi^{-1}\{0\}=M$. By abuse of notation, the map of polytopes $X=|K|\to [0,1]$
induced by $\pi$ is also denoted by $\pi$. Inductively
we will construct subdivisions $K_n$ $(n=1,2,\cdots)$ of $K$ with the following property.

\noindent{\it  There is a full subcomplex $L_n$ of $K_n$ such that $|L_n|\supset \pi^{-1}([1/n,1])$ 
and each simplex $\si\in L_n$ is contained in a $U_\al.$}

\noindent Let $\text{sd}K$ be the barycentric 
subdivision of $K$. Since the polytope $|\pi^{-1}\{1\}|=|C(M,K)|$ is compact,  there exists an $m$ such that each simplex in 
$\text{sd}^m(\pi^{-1}\{1\})$ is contained in a $U_\al$. The complex $K_1$ is defined to be $\text{sd}^mK$ 
for the smallest $m$ with the property, and $L_1$ is defined to be $\text{sd}^m(\pi^{-1}\{1\}).$
Suppose that $K_n$ and $L_n$ has been constructed.  We denote by $\text{sd}^m K_n \text{ mod }L_n$
the $m$-th barycentric subdivision of $K_n \text{ mod }L_n$. 
\begin{lemma}
For sufficiently large $m$, each simplex $\si\in N(\pi^{-1}[\frac1{n+1},1],\,\, 
{\rm sd}^m K_n  \mod \,L_n)$ is contained in a $U_\al.$
\end{lemma}

\begin{proof}
For a simplex $\si$ of $K_n$, $\si\cap L_n$ is either empty or is a face of $\si$ 
since $L_n\lhd K_n$. 
Suppose that $\si\cap L_n=:\tau$ is not empty. If $\si\notin L_n$, 
then $\si$ is the join $\tau\ast \eta $ for a face $\eta$ with 
$\eta\cap L_n=\emptyset.$
Let $pr:\,\,\si\to [0,1]$ be the simplicial map such that $pr^{-1}\{0\}=\tau$ and 
$pr^{-1}\{1\}=\eta$. Then we see by induction on $k$ that $N(\tau, \,\,\text{sd}^k\si \text{ mod }
\tau)\subset pr^{-1}([0,(\frac {\dim \si}{\dim \si+1})^k])$. By the induction hypothesis $\tau$ is contained
in a $U_\al$. Hence
for a sufficiently large $m$ $|N(\tau, \,\,\text{sd}^k\si \text{ mod }
\tau)|\subset U_\al$.  The set $|N(\pi^{-1}[\frac1{n+1},1],\,\,\text{sd}^m K_n \text{ mod }L_n
)|$ is a compact subset of $U$. It follows that for a possibly
larger $m$ each simplex of $N(\pi^{-1}[\frac1{n+1},1],\,\,\text{sd}^m K_n \text{ mod }L_n)$ 
is contained in a $U_\al$. By Lemma 3.3 (a) \cite{RS},
There is a subdivision of $\text{sd}^m K_n \text{ mod }L_n$ of which  $N(\pi^{-1}[\frac1{n+1},1],\,\,\text{sd}^m K_n \text{ mod }L_n)$ is a full subcomplex. Let $K_{n+1}$ be such a subdivision,
and $L_{n+1}$ is defined to be 
\newline $N(\pi^{-1}[\frac1{n+1},1], \,\,\text{sd}^m K_n \text{ mod }L_n)$.
\end{proof}
Let $L_\infty=\underset{n}\cup L_n$. Then $L_\infty$ is a triangulation of $U$
each simplex of which is contained in a $U_\al$. For each $\si\in L_\infty$,
choose one such $\al_0$ and define $l_\al(\si)=
\begin{cases}\si& \al=\al_0\\
0& \al\neq \al_0
\end{cases}.$  Let $x\in U$ be a point, and $s\in (\cal C_j)_x$ an element of
the stalk of the sheaf $\cal C_j$ at $x$. $s$ is the restriction
of an element $S\in C_j(K, (X-V_x))$ for a neighborhood $V_x$ of $x$ and for a triangulation
$K$ of $X$. We can assume that $V_x\subset |L_n|$ for an $n$. $S$ is a sum
$\sum a_\si \si$. Taking $V_x$ smaller if necessary, we can assume that each $\si$ contains $x$. Take a common subdivision $K'$ of $K$ and $K_n$. 
For each simplex $\tau$ of $K$ and $\eta$ of $K_n$, the intersection
$\tau\cap \eta$ is the union of interior of several simplexes of $K'$. For a simplex
$\si\in K$ with $a_\si \neq 0$, we have $\si\cap L_n=\sum \tau_m \text{ mod }
X-V_x$ with each $\tau_m$ a $j$-simplex of $K'$ since $V_x$ is an neighborhood of $x$ in $U$. 
For each $\tau_m$, let $\eta_m$ the smallest simplex of $K_n$ which contains $\tau_m$. 
 Then we define
\[l_\al(\tau_m)=\begin{cases}
\tau_m & l_\al(\eta_m)=\eta_m\\
0 & l_\al(\eta_m)=0.
\end{cases}\]
This definition is compatible with subdivisions since 
we have $\tau_m^\circ \subset \eta_m^\circ.$  
\end{proof}
Now we can prove Theorem \ref{cohom isom}. For a sheaf $\cal F$, we denote by $G\cal F$ the canonical resolution of Godement of $\cal F$. 
As in the proof of  Proposition \ref{fineness}, 
let $(K,M)$ be a semi-algebraic 
triangulation of the pair $(X,D)$ such that $M\lhd K$ and
$\pi:\, K\to [0,1]$ be the simplicial map defined on each vertex $v$ by
\[\pi(v)=
\begin{cases}
0& v\in M\\
1& v\notin M.
\end{cases}
\]  We denote by $C_D$ the polytope $|C(M,K)|$, and by $i_D$ the inclusion of
$C_D$ into $U$. 
We have the following commutative diagram with exact rows.
\[
\begin{matrix}
0&\xrightarrow{} &\Ker i_D^*&\xrightarrow{}& \Gamma(X, \cal C^*[-2m])&
\xrightarrow{\scriptsize{i_D^*}}& \Gamma(C_D, \cal C^*[-2m])&\xrightarrow{}&0\\
&&\uparrow{\text{\scriptsize{$\al_1$}}}&&\uparrow{\text{\scriptsize{$\al_2$}}}&&\uparrow{\text{\scriptsize{$\al_3$}}}&&\\
0&\xrightarrow{}&C_{2m-*}(D;\ZZ)&\xrightarrow{}&C_{2m-*}(X;\ZZ)&\xrightarrow{}&C_{2m-*}(X,D;\ZZ)&\xrightarrow{}&0
\end{matrix}
\]
See Definition \ref{def:definition of AC with inductive limit} for the definition of $C_{2m-*}(D;\ZZ)$ and $C_{2m-*}(X;\ZZ)$.
\begin{lemma}
\label{D-qis}
The map $\al_1$ is a quasi-isomorphism.
\end{lemma}
\begin{proof}
We denote by $V_n$ the polytope $\pi^{-1}[0, 1-1/n]$.  By \cite{Mu} Lemma 70.1
$D$ is a deformation retract of $V_n$, so that the natural inclusion
\[C_{2m-*}(D;\ZZ)\to C_{2m-*}(V_n;\ZZ)\] is a quasi-isomorphism. 
Since $V_n$ is compact, the natural
map
\[C_{2m-*}(V_n;\ZZ)\to \Gamma(V_n, \cal C^*[-2m])\]
is an isomorphism. We have the assertion
because $\Ker i_D^*=\underset{n}\varinjlim \Gamma(V_n, \cal C^*)$.
\end{proof}
Since $X$ is compact, the map $\al_2$ is an isomorphism. It follows that $\al_3$ is a quasi-isomorphism.
 Since  $C_D$ is a deformation retract of $U$, by Proposition \ref{fineness}
the restriction map $\Gamma(U,\cal C^*)\to \Gamma(C_D, \cal C^*)$ is a quasi-isomorphism. Hence
the natural map $\al_4:\,C_{2m-*}(X,D;\ZZ)\to \Gamma(U,\cal C^*[-2m])$ is quasi-isomorphic,
and the map $\beta:\,\Gamma(U,\cal C^*)\to \Gamma(U,G\cal C^*)$
is quasi-isomorphic since $j^*\cal C^*$ is a complex of fine sheaves.
We  have a commutative diagram
\begin{equation}
\label{q-isom}
\begin{matrix}
\Gamma(U,G\AC^*[-2m])&\xrightarrow{\text{\scriptsize{$\iota_1$}}}& \Gamma(U, G\cal C^*[-2m])\\
\uparrow{\scriptsize{r}}&&\uparrow{\text{\scriptsize{$\beta$}}}\\
\Gamma(U, \AC^*[-2m])&\xrightarrow{}&\Gamma(U,\cal C^*[-2m])\\
\uparrow{\text{\scriptsize{$\al$}}}&&\uparrow{\text{\scriptsize{$\al_4$}}}\\
AC_{2m-*}(X,D;\ZZ)&\xrightarrow{\text{\scriptsize{$\iota_2$}}}&C_{2m-*}(X,D;\ZZ)
\end{matrix}
\end{equation}
Since the maps $\iota_1$, $\beta\circ \al_4$ and $\iota_2$ are quasi-isomorphisms,
the map $r\circ \al$ is also a quasi-isomorphism. 
The face maps 
$\part_{H_i}:\,AC_*(X,D;\ZZ)\to AC_{*-2}(H_i, H_i\cap D;\ZZ)$ induce face maps
$\part_{H_i}:\,\Gamma(U, G\AC[-2m])\to \Gamma(H_i\cap U, G\AC[-2(m-1)])$. 
\begin{notation}
\label{simple complex}
For a double complex $C^{*,*}$, we denote by $s(C^{*,*})$ the simple complex associated to $C^{*,*}.$
\end{notation} 
We see that
the complex $AC^*(U,{\bf H})$ is quasi-isomorphic to the simple complex associated to
the double complex 
\[\AC^{p,q}(U)=\underset{\sharp I=p}\oplus \Gamma(H_I\cap U, G\AC[-2(m-p)]^q).\]
 The cohomology of the complex $s(\AC^{*,*}(U))$ 
is equal to $H^*(U,{\bf H};\ZZ)$.
\end{proof}

\section{The duality}
In this section we will show that the duality between the de Rham cohomology
and the singular cohomology can be described via integral on admissible chains.
\begin{definition}
\label{res}
Let $\varphi$ be a smooth $p-$form on $X$ with logarithmic singularity along $\bf H$.
\begin{enumerate}
\item Let $H_\al$ be a codimension one face of $X$. The Poincare residue of $\varphi$ at $H_\al$, which
is denoted by $\res_{\{\al\}}\varphi$, is defined as follows. If $z=0$ is a local equation of $H_\al$
and $\varphi=\frac{dz}z\we \psi+\eta$ where $\eta$ does not contain $dz$, then $\res_{\{\al\}}\varphi
=\psi|_{H_\al}.$ Note the different sign convention from the one defined in \cite{Gr}.

\item 
For a subset $I=\{i_1<\cdots <i_k\}$ of $T$, the residue of $\varphi$ at $H_I$,
which is denoted by $\res_I (\varphi)$, is defined by the succession of Poincare residues
\[\res_{H_{i_1}}\circ \cdots \circ\res_{H_{i_k}}(\varphi).\]

\item For an element $\ga$ of $AC_{p-\sharp I}(H_I,D\cap H_I;\ZZ)$, the integral
\[(2\pi i)^{\sharp I}\int_\ga \res_I(\varphi)\]
is denoted by $(\ga,\varphi)$.

\end{enumerate}
\end{definition} 

\begin{proposition}
\label{C-S}
Let $X$ be a smooth projective variety over $\C$, $D$ a closed subset of $X$, $U=X-D$, ${\bf H}=H_0\cup \cdots \cup
H_t$ a strict normal crossing divisor on $X$ and $\varphi$ be a smooth $p$-form on $U$ with compact support 
and with logarithmic singularity along $\bf H$. 
\begin{enumerate}

\item  For an admissible $p-\sharp I$-simplex $\si$ on $H_I$, the integral
\[\int_\si\res_I(\varphi)\]
converges absolutely.

\item Let $\ga$ be an element of $AC_{p-\sharp I+1}(H_I,D\cap H_I;\ZZ)$. Under the notation of Notation \ref{res},  we have an equality
\begin{equation}
\label{Cauchy-Stokes}
(\part_{\bf H} \ga, \varphi)=(\delta\ga, \varphi)-(-1)^{\sharp I}(\ga, d\varphi).
\end{equation}
\end{enumerate}
\end{proposition}
\begin{remark}
\label{counter example}
\begin{enumerate}
\item The convergence of integrals as in the assertion (1) fails in general for $C^\infty$-chains.
For example, consider the chain $D$ in $\CC^2$ given  by 
\[\{(x,y)\,|\,(x,y)\in \RR^2,\,1\leq x\leq 2,\,\,e^{-\frac{1}{x-1}}\leq y\leq e^{-1}\}.\]
If we set the face of $\CC^2$ to be $\{z_1=0\}\cup \{z_2=0\}$, then the chain $D$ meets the faces 
of $\CC^2$ properly, but the integral $\dis \int_D\frac{dz_1}{z_1}\we 
\frac{dz_2}{z_2}$ diverges. $D$ is not $\delta$-admissible
in the sense that the boundary of $D$ does not meet the face $\{z_2=0\}$
properly, but adding some chains in the imaginary direction we obtain a 
$\delta$-admissible chain.  

\item A consequence of Proposition \ref{C-S} is that an element of $AC_*(U,\bf H)$
defines a normal current of intersection type along $\bf H$. See \cite{Ki} and \cite{KL}
for the definition of normal currents of intersection type along $\bf H$.
\end{enumerate}
\end{remark}
\begin{proof} First we consider (1). The same proof as that of \cite{part I} Theorem 4.4 works, but we need some modification.
The proof of the convergence is reduced to showing ``allowability'' of a certain semi-algebraic 
set. The argument given in Section 4 of \cite {part I} should be modified as follows. We need to consider the following
situation: Let $\CC^n\times\CC^k$ with coordinates $(z_1,\cdots, z_{n+k})$. Let $ H_i, 1=1,\cdots, n$ be the coordinate
hyperplanes. A face of $\CC^n\times \CC^k$ is a subset of the form $H_I=\underset{i\in I}\cap H_i$
where $I$ is a subset of $\{1,\cdots, n\}$. A closed semi-algebraic subset $A$ of  $\CC^n\times \CC^k$ is said to be
admissible if for any face $H_I$, an inequality
\[\dim(A\cap H_I)\leq \dim A-2\sharp I\]
holds. We need to show the following:

\quad

\noindent {\bf Claim}. {\it Let $A$ be a compact admissible semi-algebraic subset of $\CC^n\times \CC^k$ of dimension
$h$. If $\varphi$ is an $h$-form on $\CC^n\times \CC^k$ of the shape
\[\frac{dz_1}{z_1}\we \cdots \we \frac{dz_n}{z_n}\we \eta\]
where $\eta$ is a smooth $(h-n)$-form on $\CC^n\times \CC^k$, then the integral
\[\int_A\varphi\]
converges absolutely.} 

\quad

\noindent We divide the complex plane $\CC$ in to four sectors
\[S_i=\{z=re^{i\theta}|0\leq r<\infty,\,\,(i-2)\pi/4\leq \theta\leq i\pi/4\},\,\,i=0,1,2,3.\]
and make a coordinate change as follows: For $z=x+iy$, we set $r=|z|,\,\,\tau=y/x$. We have
\[x=\frac{r}{\sqrt{\tau^2+1}},\quad y=\frac{r\tau}{\sqrt{\tau^2+1}}.\]
We have a continuous semi-algebraic map
\[\pi:\,\,\RR_{\geq 0}\times [-1,1]\to S_0,\quad (r,\tau)\mapsto (x,y).\]
We also have equalities
\[ \frac{dz}{z}=\frac{dr}{r}+i\frac{d\tau}{\tau^2+1},\,\,\frac{dz}{z}\we d\bar{z}=\frac{-2(\tau+i)}{(\tau^2+1)^{3/2}}
dr\we d\tau.\] 

The form $\varphi$ is the sum of  forms of type
\begin{equation}
\label{original form}
f \frac{dz_1}{z_1}\we \cdots \we \frac{dz_n}{z_n}\we \underset{k\in R}\bigwedge d\bar{z}_k
\we( du_1\we \cdots \we du_{h-n-|R|}).
\end{equation}
Here $R$ is a subset of $\{1,\cdots, n\}$, $\{u_1,\cdots, u_{h-n-|R|}\}\subset \{x_{n+1},y_{n+1},\cdots,
x_{n+k}, y_{n+k}\}$ and $f$ is a smooth function on a neighborhood of $A$. 

Writing $\tilde{\CC}_+=\RR_{\geq 0}\times [-1,1]$, let
\[\pi=\pi_{\al_1,\cdots,\al_n}\times (\text{identity}):\,\tilde{\CC}_+^n\times \CC^k\to (S_{\al_1}\times \cdots \times S_{\al_n})
\times \CC^k\]
be the product of the maps $\pi:\,\,\tilde{\CC}_+\to S_{\al_i}$ and the identity of $\CC^k$, still denoted by the same letter. 
The pull-back of the form (\ref{original form}) is the sum of the forms
\[
f'\varphi_{P,Q,R}\we (  du_1\we \cdots \we du_{h-n-|R|})
\]
where
\begin{equation}
\label{decomposed form}
\varphi_{P,Q,R}=\underset{i\in P}\bigwedge \frac{dr_i}{r_i}\we \underset{j\in Q}\bigwedge d\tau_j\we 
\underset{k\in R}\bigwedge (dr_k\we d\tau_k)
\end{equation}
where $(P,Q)$ varies over partitions of $\{1,\cdots, n\}-R$, and $f'$ is a smooth function.

Consider the map
\[\tilde{\CC}_+^n\times \CC^k\to \RR^P_{\geq 0}\times [-1,1]^Q\times \tilde{\CC}_+^R\times \RR^{h-n-|R|}\]
given by the product of the maps
\[
\begin{array}{ll}
\tilde{\CC}^P_+\to \RR^P_{\geq 0},& (r_i,\tau_i)\mapsto r_i,\\
\tilde{\CC}^Q_+\to [-1,1]^Q,& (r_j,\tau_j)\mapsto \tau_j\\
\tilde{\CC}^R_+\to \tilde{\CC}^R_+,& \text{the identity map}\\
\CC^k\to \RR^{h-n-|R|}, &(z_{n+1},\cdots, z_{n+k})\mapsto (u_1,\cdots, u_{h-n-|R|}).
\end{array}
\]
Taking the product of these maps we obtain a map
\[q:\,\,\tilde{\CC}_+^n\times \CC^k\to (\RR)^P\times (\RR)^Q\times (\RR^2)^R\times \RR^{h-n-|R|}.\]
We need to show the absolute convergence of the integral

\[\int_{\pi^{-1}(A)}\varphi_{P,Q,R}\we (  du_1\we \cdots \we du_{h-n-|R|}).\]
Let $\varphi'_{P,Q,R}$ be the form on $\RR^P\times \RR^Q\times (\RR^2)^R$ given by the same formula
as (\ref{decomposed form}). Applying Proposition 2.5 \cite{part I} to the set $\pi^{-1}(A)$ and the map $q$,
we are reduced to showing the absolute convergence of the integral
\[\int_{q\pi^{-1}(A)} \varphi'_{P,Q,R}\we (du_1\we\cdots \we du_{h-n-|R|}).\]
The argument after this is the same as that in Section 4 of \cite{part I}. Note that since the form we consider 
has a compact support contained in $X-D$, we can assume that $A\cap D=\emptyset$. So the argument is simpler
than the case of loc. cit.

Next we prove the assertion (2).  First we assume that the set $I$ is empty. 
The proof  is in the same line as that of Theorem Theorem 4.3 \cite{part II} with 
some modification. Let ${\bf H}_c$ be the union of higher codimensional faces i.e.
\[{\bf H}_c=\underset{i\neq j}\cup( H_i\cap H_j).\]
 First we consider the case where
$|\ga|\cap {\bf H}_c=\emptyset$, and then prove the general case by a limit argument.
So suppose that $|\ga|\cap {\bf H}_c=\emptyset.$  Let us write $\ga=\sum a_\si \si$. For a codimension one face
$H_k$, set $\dis \ga_k=\sum_{\si\cap H_k\neq\emptyset} a_\si\si.$ Then we have
\[\ga=\sum_k\ga_k+\ga'\]
where $|\ga'|\cap \bf H=\emptyset$. since $|\ga|\cap{\bf H}_c=\emptyset$, $|\ga_k|\cap H_j=\emptyset$
for $j\neq k$, and so that $\ga_k\in AC_r(K,D)$ for each $k$. Here we denote $r=p-\sharp I+1$ for short.
It suffices to prove $(2)$ for each $\ga_k$. So we consider the case
where $|\ga|\cap {\bf H}\subset H_0=H$. For a simplex $\tau\in \ga\cap H$, let
\[\ga^{(\tau)}=\sum_{\si\cap H=\tau}a_\si\si.\]
Since $K$ is a good triangulation, the intersection of each simplex $\si$ of $K$ with $F$ is a simplicial face of $\si$. So we have
\[\ga=\sum_{\tau\in \ga\cap H}\ga^{(\tau)}.\]
By Proposition 4.11 of \cite{part II}, we have $\ga^{(\tau)}\in AC_r(K,D)$. It suffices to prove $(2)$ for each $\ga^{(\tau)}$. So we assume that $\ga=\ga^{(\tau)}$ for a simplex $\tau\in H$. 
By taking sufficiently fine subdivision, we can assume that $|\ga|$ is contained in a coordinate
neighborhood $V$ of $X$ on which $H$ is defined by an equation $z_H=0$ for a function $z_H$. 
Under the comparison isomorphism of de Rham and singular cohomology, the class of $H\cap V$
in $H^2_{dR, H}(V,\CC)$ is equal to $\delta (\frac1{2\pi i}\frac{dz_H}{z_H})$ where
\[\delta:\,H^1_{dR}(V-H)\to H^2_{dR,H}(V)\]
is the boundary map. For $\epsilon>0$, let $\rho_\e:\,\CC\to [0,1]$ be a $C^\infty$-function
such that
\[\rho_\e(z)=
\begin{cases}
0& |z|<1/2\e \\
1& |z|>\e
\end{cases}
\]
and set $c_\e=\rho_\e(z_H)\frac1{2\pi i}\frac{dz_H}{z_H}.$ Put $C=|C(H,K)|$. For $\e$ sufficiently small 
we have $c_\e=\frac1{2\pi i}\frac{dz_H}{z_H}$ on $C$, and $\delta (\frac1{2\pi i}\frac{dz_H}{z_H})$
and $dc_\e$ defines the same class of $H^2_{dR}(V,C)\simeq H^2_{dR}(V,V-H)=
H^2_{dR,H}(V)$. We use the form $dc_\e$ as our Thom cocycle. 

\noindent By the Stokes formula, we have
\[\int_\ga d(\rho_\epsilon\varphi)-\int_{\delta\ga}\rho_\epsilon \varphi=0\]
and so
\[ \int_\ga d\rho_\epsilon\we \varphi=\int_{\delta\ga}\rho_\epsilon \varphi-\int_\ga \rho_\epsilon d\varphi.\]
By Lebesgue convergence theorem the right hand side of this equality converges to the  right hand side of
(\ref{Cauchy-Stokes}) as $\e\to 0.$  So we need to show that 
\begin{equation}
\label{limit}
\underset{\e\to 0}\lim \int_\ga d\rho_\epsilon\we \varphi=2\pi i\int_{\part \ga}\res \varphi.
\end{equation}
The rest of the argument is the same as the proof of Proposition 4.7 \cite{part II}.
The limit argument which is necessary to prove the general case is the same as the one given 
in Section 4.4 of \cite{part II}.  We consider a general case. Let $\ga$ be an element of $AC_{p-\sharp I+1}(H_I,D\cap H_I;\ZZ)$
with $I=\{\al_0<\cdots <\al_k\}$. If $z_i=0$ for $0\leq i\leq k$ are local equations of $H_{\al_i}$, and
$$\varphi=\frac{dz_k}{z_k}\we \frac{dz_{k-1}}{z_{k-1}}\we \cdots \we \frac {dz_0}{z_0}\we \psi+\eta$$ where
$\eta$ does not conatin $dz_k\we \cdots \we dz_0$, then $\res_I\varphi=\psi|_{H_I}$. 
If
 $$I'=\{\al_0<\cdots <\al_{i-1}<\beta<\al_i<\cdots <\al_k\},$$ 
then we have
$$2\pi i\int_{\part_{H_\beta}\ga}\res_{\{\beta\}}(\res_I \varphi)=\int_{\delta\ga}\res_I\varphi-\int_\ga d(\res_I\varphi) .$$
Since we have $\res_{I'}\varphi=(-1)^i\res_{\{\beta\}}(\res_I \varphi)$ and
$d(\res_I\varphi)=(-1)^{\sharp I}\res_I d\varphi$, we have the assertion.
\end{proof}

\begin{theorem}
\label{duality}
The map
\[AC^*(U,{\bf H})\otimes A^*_c(U)(\log {\bf H})\to \C,\quad \ga\otimes \varphi\mapsto (-1)^{\e_1(\ga)}(\ga,\varphi)\]
is a map of complexes, and 
induces the duality pairing
\[H^*(U,{\bf H};\QQ)\otimes H^*_{dR}(X-{\bf H},D;\C) \to \C.\]
Here the function $\e_1(\ga)$ is defined as follows. If $\ga$ is on the face $H_I$, then $\e_1(\ga)=\e_1(\dim\ga, \sharp I)$
where $\e_1(x,y)=\frac{x(x+1)}2+xy$. 
\end{theorem}
\begin{proof} The fact that the above map is compatible with the differential is a consequence of 
Proposition \ref{C-S}. For the second assertion, first we reduce the problem to the case where $D=\emptyset.$ There exists a sequence of
blow-ups
\[\pi:\,\widetilde{X}\to X\]
such that $\pi^{-1}({\bf H}\cup D)$ is a simple normal crossing divisor and $\pi$ induces  an isomorphism
$\pi^{-1}(U)\to U$. 
\begin{lemma}
\label{blow up}
There is a quasi-isomorphism 
$$\pi_*:\,\,AC_*(\widetilde{X},\pi^{-1}(D))\to AC_*(X,D)$$
which is compatible with the pairing $(\bullet,\bullet)$: for $\varphi\in
A^p_c(U)(\log {\bf H})$ and $\ga\in AC_p(\widetilde{X},\pi^{-1}(D))$, we have
$(\ga,\varphi)=(\pi_*(\ga),\varphi)$.
\end{lemma}
\begin{proof} Let $K$ be a good triangulation of $\widetilde{X}$. For a $j$-simplex
$\si\in K$,  there is a good triangulation  $K'$ of $X$ such that $\pi(\si)$ is the union
of the interior of some simplexes of $K'$ by Theorem  \ref{thm:semi-algebraic triangulation}.
We define $\pi_*(\si)$ to be the sum of $j$-simplexes of $K'$ contained in $\pi(\si)$.
The orientation of each simplex is defined by the compatibility with that of $\si$. 
By (\ref{q-isom}) we have a commutative diagram
\[
\begin{matrix}
\Gamma(U, G\AC^*[-2m])&\xrightarrow{=}&\Gamma(U,G\AC^*[-2m])\\
\uparrow{\text{\scriptsize{$\widetilde{r\circ \al}$}}}&
&\uparrow{\text{\scriptsize{$r\circ \al$}}}\\
AC_{2m-*}(\widetilde{X},\pi^{-1}(D);\ZZ)&\xrightarrow{\text{\scriptsize{$\pi_*$}}}&AC_{2m-*}(X,D;\ZZ)
\end{matrix}
\] From this we see that $\pi_*$ is a quasi-isomorphism.
The compatibility of $\pi_*$ and the pairing $(\bullet,\bullet)$ follows from the definition.
\end{proof}
 By Lemma \ref{blow up} we can replace $X$ resp. $\bf H$ with $\widetilde{X}$
resp. the proper transform of $\bf H$.
Let $D_j\,(j=0,\cdots, s)$ be the irreducible components of $D$. We denote
the index set $\{0,\cdots, s\}$ by $S$. For a subset $J$ of $S$ we denote $\underset{j\in J}\cap D_j$
by $D_J$.

\noindent Let ${}_{dR}C^{*,*}$ be the double complex defined by 
\begin{equation}
\label{de Rham double}
{}_{dR}C^{pq}=\underset{p=\sharp J}\oplus A^q(D_J)(\log {\bf H}),\ d_1=d_D,\,d_1=d_{dR}\,(\text{the exterior derivative})
\end{equation}
where the differential $d_D$ is defined as follows: for an element $\omega=\underset{J\subset S}\oplus \omega_J$ of
${}_{dR}C^{*,*}$, we have
\[(d_D\omega)_{(\al_0<\cdots<\al_k)}=\sum_{i=0}^k(-1)^i\iota_{\al_i}^*\omega_{(\al_0<\cdots \widehat{\al_i}\cdots<\al_k)}.\]
Here the map $\iota_{\al_i}:\,D_{(\al_0<\cdots <\al_k)}\to D_{(\al_0<\cdots \widehat{\al_i}\cdots <\al_k)}$ is the inclusion map.
\begin{lemma}
\label{compact-total}
We have an isomorphism $H^j(A^*_c(U)(\log {\bf H}))\simeq H^j(s({}_{dR}C^{*,*}))$.
\end{lemma}
\begin{proof}
We denote by $\cal A^*(D_J)(\log \bf H)$ the complex of sheaves of $C^\infty-$ forms on $D_J$ with
logarithmic singularity along  $\bf H$. 
${}_{dR}\cal C^{*,*}$ be the double complex of sheaves
defined by
\[{}_{dR}\cal C^{pq}=\underset{p=\sharp J}\oplus i_*\cal A^q(D_J)(\log {\bf H})\]
where the map $i$ is the inclusion of $D_J$ into $X$, and the differentials are defined in the same way as (\ref{de Rham double}).
Let 
$\pi:\,\,X\to [0,1]$ be the map defined in (\ref{projection}), and set 
$V_n=\pi^{-1}([1/n,1])$ for $n\geq 1$. Let $\cal A^*_{V_n}(X)(\log {\bf H})$ be the complex of sheaves on $X$ of $C^\infty-$forms with log singularity along $H$ whose supports are contained
in $V_n$. Let $\cal A_c^* (X)(\log {\bf H})
=\underset{n}\varinjlim \,\cal A^*_{V_n}(X)(\log {\bf H})$. Then we have the equality
$A^*_c(U)(\log {\bf H})=\Gamma( X, \cal A_c^* (X)(\log {\bf H}))$. We will show that the inclusion
\[\cal A^*_c(U)(\log {\bf H})\hookrightarrow s({}_{dR} \cal C^{*,*})
\] is a quasi-isomorphism of complexes of sheaves. Then since these are complexes of fine sheaves, this inclusion
induces an isomorphism of cohomology groups
\[H^i(A^*_c(U)(\log {\bf H}))\to H^i(s({}_{dR}C^{*,*} ))\simeq H^i(X-\bf H,D,\CC).\]
The problem is local,  so let $x$ be a point on $D$. Near $x$ $X$ is isomorphic to
$U_x:=B^J\times B^K $ with $x$ the origin  where $B$ is the open unit disc in $\CC$, and $D=\{z_1\cdots z_{\sharp J}=0\}$. Suppose first that $J=\{1\}$. The complex 
\[\text{Cone} (i^*:\,\cal A^*(U_x)(\log {\bf H})
\to A^*(D_1)(\log {\bf H}))\]
 where $i$ is the inclusion map of $D_1$ into $U_x$, is quasi-isomorphic
to $\text{Ker}\,i^*=:\cal A_0^*$. We need to show that the quotient $\cal A_0^*/\cal A^*_c(\log {\bf H})$
is acyclic. We proceed as the proof of Poincare Lemma.  Let $F:\,[0,1]
\times U_x\to U_x$ be the deformation retract defined by $(t, (z_1,\cdots, z_{1+K}))
\mapsto ((1-t)z_1,z_2,\cdots, z_{1+K})$.  There are maps
\[g_0,\,g_1:\,U_x\to [0,1]\times U_x,\, g_i(p)=(i,p).\]
Let $h:\, A^{j+1}([0,1]\times U_x)
\to A^j(U_x)$ be the map given by the integration along the fibers of the projection
$[0,1]\times U_x\to U_x.$ Then we have the equality
\begin{equation}
\label{homotopy}
g_0^*-g_1^*=\pm(dh-hd).
\end{equation}  Suppose that $\phi$ is a cocycle of  $A_0^*(U_x)/A^*_c(U_x)(\log {\bf H})$ i.e.
$d\phi\in A^*_c(U_x)$. Taking $U_x$ sufficiently small we can assume that $d\phi=0$.
By applying \ref{homotopy} to the form $F^*\phi$ we see that
\[\phi=\pm dhF^*\phi .\] Since the restriction of $F$ to $D_1$ is constant with respect
to $t$, we see that $hF^*\phi\in A_0^j(U_x)$. This concludes the proof for $J=\{1\}$. 
we can proceed by induction on $\sharp J$. Suppose the assertion holds for the case
$\sharp J=k$ and consider the case $\sharp J=k+1$. We have the equality
\[\begin{array}{ll}
&s({}_{dR}C^{*,*})\\
= &
 \text{Cone}\bigl(i_1^*:\,s(\underset{K\subset J-\{1\}}\bigoplus \cal A^*(D_K))\to 
s(\underset{K\subset J-\{1\}}\bigoplus \cal A^*(D_{K\cup \{1\}}))\bigr).
\end{array}\]
By induction hypothesis, This is quasi-isomorphic to the complex
\[\text{Cone}\bigl(i_1^*:\,\cal A_c^*(X-\underset{j\neq 1}\cup D_j) \to
\cal A_c^*(D_1-\underset{j\neq 1}\cup D_j)\bigr).\]  By the same argument as in the proof of the
case where $\sharp J=1$, we can show that this complex is quasi-isomorphic to
$\cal A_c^*(U_x)$. 
\end{proof}
The complex $C_*(X)/C_*(D)$ is quasi-isomorphic to the simple complex associated to the
double complex $C(1)_{*,*}$ defined by
\[C(1)_{a,b}=\underset{J\subset S, \sharp J=a}\oplus C_b(D_J).\]
Here the second differential is $\delta$, the topological boundary, and 
 the first differential $\part_D$ is defined by 
\begin{equation}
\label{part-D}
\part_D( \ga_{(\al_0<\cdots <\al_p)})=\sum_{i=0}^p(-1)^i \iota_{\al_i}(\ga_{(\al_0<\cdots <\al_p)})
\end{equation}
where $\iota_{\al_i}:\,H_{\{(\al_0,\cdots, \al_p)\}}\hookrightarrow H_{\{\al_0,\cdots,\widehat{\al_i},
\cdots,\al_p\}}$ is the inclusion map. 
By Proposition \ref{prop: moving lemma},  $C(1)_{*,*}$ is quasi-isomorphic to the complex $C(2)_{*,*}$
defined by
\[C(2)_{a,b}=\underset{J\subset S, \sharp J=a}\oplus AC_b(D_J).\]
Hence the complex $AC^*(U,{\bf H})$ is quasi-isomorphic to the simple complex $s(C(3)^{*,*})$ associated
to the double complex $C(3)^{*,*}$ defined by
\[C(3)^{p,q}=\underset{\substack{I\subset T,\,\sharp I=p\\
J\subset S,\,\sharp J+i=2m-2p-q}}\oplus AC_i(H_I\cap D_J)),\,\,d_1=\part_{\bf H},\,d_2=\part_D+(-1)^{\sharp J}\delta.\]
The  differential of $s(C(3)^{*,*})$ is equal to
\[\part_{\bf H}+(-1)^{\sharp I}\part_D+(-1)^{\sharp I+\sharp J}\delta\]
on $AC_*(H_I\cap D_J)$.  
Let ${}_{top}C^{p,q}$  be the double complex defined by
\[{}_{top}C^{pq}=\underset{p=-\sharp J, q=2m-\sharp I-i}\oplus AC_i(H_I\cap D_J),\, d_1=\part_D,\,d_2=\part_{\bf H}+(-1)^{\sharp I}\delta\]
The map $s(C(3)^{*,*})\to s({}_{top}C^{*,*})$ defined by $\ga\mapsto (-1)^{\sharp I\sharp J}\ga$ for $\ga\in AC_*(H_I\cap D_J)$ 
is an isomorphism of complexes.
By Proposition \ref{C-S}, the pairing of Notation \ref{res} (2) induces a pairing
\[s({}_{top}C^{*,*})\otimes s({}_{dR}C^{*,*}) \to \CC[-2m],\,\,\ga\otimes \varphi\mapsto (-1)^{\e_1(\ga)+\e_2(\ga)}(\ga,\varphi)\]
which is a map of complexes.  Here the target is the complex $\CC$ concentrated in degree $2m$,
and the function $\e_2$ is defined as follows. If $\ga$ is on $D_J\cap H_I$, then
$\e_2(\ga)=\e_2(\dim \ga-\sharp I, \sharp J)$ where $\e_2(x,y)=xy+\frac{y(y+1)}2.$
We need to show that this pairing induces perfect pairing on cohomology.
The above pairing induces a map of double complexes
\[{}_{top}C^{pq}\to ({}_{dR}C^{-p\,2m-q})^\vee\]
which is compatible with the differentials up to sign.
There are spectral sequences
\[{}_{top}E_0^{pq}=\underset{p=-\sharp J, q=2m-\sharp I-j}\oplus AC_j(H_I\cap D_J)\Rightarrow H^{p+q}(s({}_{top}C^{*,*}))\]
and
\[{}_{dR}E_0^{pq}=
\underset{p=\sharp J}\oplus A^q(D_J)(\log {\bf H})\Rightarrow H^{p+q}(s({}_{dR}C^{*,*})).\]
The above pairing induces a map of spectral sequences 
\[{}_{top}E_r^{pq}\to ({}_{dR}E_r^{-p\,2m-q})^\vee\]
compatible with the differentials up to sign. So it suffices to show that the map
\[{}_{top}E_1^{pq}\to ({}_{dR}E_1^{-p\,2m-q})^\vee\]
is an isomorphism for each $(p,q).$ This follows from the assertion of Theorem \ref{duality} for each $D_J$
since 
$${}_{dR}E_1^{-p\,2m-q}=\underset{p=-\sharp J}\oplus H^{2m-q}_{dR}(D_J-{\bf H})$$
and 
\[{}_{top}E_1^{p\,q}=\underset{p=-\sharp J}\oplus H^{2(m+p)-(2m-q)}( D_J, {\bf H};\ZZ).\]

Until the end of the proof of Theorem \ref{duality} we assume that $D=\emptyset.$

\begin{lemma}
\label{reflexivity}
\begin{enumerate}

\item Let $K$ be a complex and $L$ be a full subcomplex of $K$. Let $K'$ be a derived subdivision
of $K$ mod $L\cup C(L,K)$. Then $|N(C(L,K), K')|$ is a regular neighborhood of $|C(L,K)|$.

\item Suppose that $L$ is a full subcomplex of $K$. If $K'$ is a subdivision of $K$ inducing a 
subdivision $L'$ of $L$, then we have $L'\triangleleft K'$.
\end{enumerate}
\end{lemma}
\begin{proof} (1). Since $L$ is a full subcomplex of $K$, we have
$L=C(C(L,K),K)$, and the assertion follows from the definition of a regular 
neighborhood.

(2). This is  \cite{RS} Lemma 3.3 (b).
\end{proof}
\begin{lemma}
\label{res of neighbor}
Let $K$ be a complex and let $L_1$ and $L_2$ subcomplexes of $K$ such that $L_2$ and $L_1\cup L_2$ are full subcomplexes of $K$. Then $N(L_1, K)\cap L_2$
equals $N(L_1\cap L_2, L_2)$.
\end{lemma}

\begin{proof}
Suppose that $\tau,\eta\in K$ such that $\tau\prec \eta$, $\tau\in L_2$ and $\eta\cap L_1\neq \emptyset.$ By the assumption $\si:=\eta\cap (L_1\cup L_2)$
is a face of $\eta$. If $\si\in L_2$ we are done since $\tau\prec \si$ and $\si\cap L_1\neq 
\emptyset$. If $\si\in L_1$ there is a vertex $v\in \si\cap L_2$ since $\eta\cap
L_2\neq \emptyset$. The simplex spanned by $\tau$ and $v$ is in $L_2$ since $L_2$ is a full subcomplex
of $K$.
\end{proof}
\begin{proposition}
\label{isom}
Let $(K,L)$ be a good triangulation of the pair
 $(X,\bf H)$ such that $|N(L,K)|$ is a regular neighborhood of $\bf H$ in $X$. 
Then the natural inclusion $C_*(C(L,K);\QQ)\to AC^{2m-*}(X,\bf H;\QQ)$ is a quasi-isomorphism. 
\end{proposition} 
\begin{proof}  Recall that we set ${\bf H}=H_0\cup \cdots \cup H_t$. We proceed by induction on $t$.  When $t<0$ i.e. $\bf H$ is empty there is nothing to prove. 
Assume that the assertion holds for $t=k$ and consider the case where $t=k+1$. We denote $H_0\cup \cdots \cup H_k$ by
$F_1$ and $H_{k+1}$ by $F_2$. We set $L_1=K\cap F_1$ and
$L_2=K\cap F_2$. Since $K$ is a good triangulation, we have 
 $L_1\lhd K$, $L_2\lhd K$ and $L_1\cup L_2\lhd K$. 

For a complex $K$ and a subcomplex $L$ of $K$, a derived subdivision
$K'$ of $K$ mod $L\cup C(L,K)$ is called a {\it derived of }$K$ {\it near} $L$. cf. 3.5 
\cite{RS}. If $L\lhd K$, then the polytope $|N(L,K')|$ is a regular neighborhood
of $|L|$ in $|K|$.
\begin{lemma}
\label{Lefschetz duality}
Let $K'$ be a derived of $K$ near $L_1$, and put $N_1=N(L_1,K')$ and $C_1=C(L_1,K')$. 
Let $\mu_{C_1}\in C_*(C_1,\part C_1)$ be the fundamental chain of $C_1$. Then the map
\[\cap \mu_{C_1},\, C^*(C_1,\dot{N_1};\QQ)\to C_*(C_1, \QQ)\]
is quasi-isomorphic. Here we denote $N_1\cap C_1$ by $\dot{N_1}$.
\end{lemma}
\begin{proof}
 By Lemma \ref{reflexivity} 
and Proposition 3.10 \cite{RS} $C_1:=C(C(L_1,K),K')$ is a manifold with the boundary
$\dot{N_1}$.  The assertion follows from Lefschetz duality Theorem 3.43 \cite{Hat}.
\end{proof}

\begin{notation}
\label{induced subdivision}
 Let $K$ be a complex and $L$ be a subcomplex of $K$. If $K'$ is a subdivision
of $K$, the subdivision of $L$ induced by $K'$ is denoted by $L'$. Similar notations will be used
for other superscripts than ${}'$.
\end{notation}
Let $\text{sd}C_1$ be the barycentric subdivision of $C_1$.  
We see that 
\begin{equation}
\label{full subity}
\text{sd}(L_2'\cap C_1)\cup \sd\dot{N_1} \lhd \sd C_1.
\end{equation} 
Let $C_1^{(2)}$ be a derived 
of $\sd C_1$ near $\sd(L_2'\cap C_1)$, and set
$N_2:=N(\sd (L_2'\cap C_1), C_1^{(2)})$. 
A simplex $\si$ of $N_1$ which meets $L_1$ but not contained in $L_1$ is of the
form $\tau_1\ast \tau_2$ where $\tau_1\in L_1$ and $\tau_2\cap L_1=\emptyset$, since $L_1\lhd K$. In $\dot{N_1}^{(2)}$ $\tau_2$ is replaced with a complex $T$. Let ${N_1}^{(2)}$ be the subdivision
of $N_1$ obtained by replacing each such simplex $\si$ with $\tau_1\ast T$. The union
$K^{(2)}:=C_1^{(2)}\cup N_1^{(2)}$ is a subdivision of $K$, and we have $N_1^{(2)}\cup N_2=N(L_1\cup L_2^{(2)},K^{(2)})$. 
\vskip 1cm

\begin{center}

\includegraphics[width=10cm]{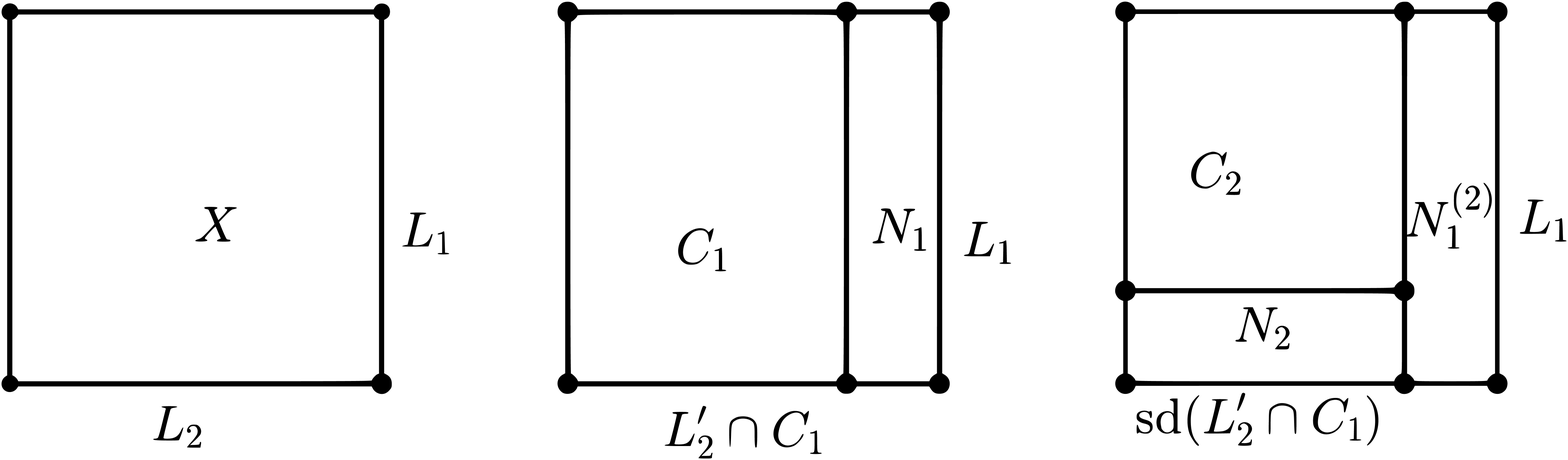}

\end{center}
\vskip 1cm

\begin{lemma}
\label{n2l2}
\begin{enumerate}
\item  Let
$C_2:=C({\rm sd}(L_2'\cap C_1), C_1^{(2)};\QQ)$ and set $\mu_{C_2}\in C_*(C_2, \part C_2;\QQ)$  be the fundamental cycle of $C_2$. Then the map
\[\cap \mu_{C_2}:\, C^*(K^{(2)},(N_1)^{(2)}\cup N_2;\QQ)\to C_*(C_2, \QQ)\]
is quasi-isomorphic.

\item The map
\[i^*:\,C^*(N_2, \dot{N_1}^{(2)}\cap N_2;\QQ)\to C^*({\rm sd}(L_2\cap C_1), (\dot{N_1}\cap L_2)^{(2)};
\QQ)\]
given by the restriction is a quasi-isomorphism. 
\end{enumerate}
\end{lemma}
\begin{proof} 
(1). By the construction
we have 
\[C^*(K^{(2)},(N_1)^{(2)}\cup N_2;\QQ)=C^*(C_2^{(2)}, \part C_2^{(2)};\QQ).\]
The assertion follows from Lefschetz duality Theorem 3.43 \cite{Hat}.

(2). By (\ref{full subity}) and Lemma \ref{reflexivity} (2) we see that 
\[{\rm sd}(L_2'\cap C_1)\cup \dot{N_1}^{(2)} \lhd C_1^{(2)}.\]
Hence by Lemma \ref{res of neighbor} we see that 
$$\dot{N_1}^{(2)}\cap N_2=N(\sd (L_2\cap \dot{N_1}), \dot{N_1}^{(2)})$$
 and so  $|\dot{N_1}^{(2)}\cap N_2|$ is a regular neighborhood of
$F_2\cap |\dot{N_1}|$ in $|\dot{N_1}|$. The assertion follows from this, since 
a polytope is a deformation retract of its regular neighborhood .
\end{proof}
 
Let $K^{(3)}$ be a derived of $K^{(2)}$ near $L_1$, and set
${{\cal N}_1}:=N(L_1, K^{(3)})$. We see that $|{{\cal N}_1}|$ is a regular neighborhood of $F_1$
in $X$. We denote
${{\cal C}_1}=C(L_1, K^{(3)})$. 
Since $|N_1|$ and $|{{\cal N}_1}|$ are regular neighborhoods of $F_1$, the map 
\[C^*(K^{(3)},{N_1};\QQ)\to C^*(K^{(3)},{{\cal N}_1};\QQ)\]
induced by the inclusion ${{\cal N}_1}\to N_1$ is quasi-isomorphic. By Lemma \ref{res of neighbor}
we have $ N_1\cap L_2'=N(L_1\cap L_2', L_2')$ and $ {\cal N}_1\cap L_2^{(3)}=
N(L_1\cap L_2^{(3)}, L_2^{(3)})$, so that $ |N_1\cap L_2'|$ and
$ |{\cal N}_1\cap L_2^{(3)}|$ are regular neighborhoods of $F_1\cap F_2$ in $F_2$.
 Hence the natural map
\begin{equation}
\label{neighbor comparison}
C^*\bigl(L_2^{(3)},({N_1}\cap L_2)^{(3)};\QQ\bigr)\to C^*\bigl(L_2^{(3)},({{\cal N}_1}\cap L_2)^{(3)};\QQ\bigr)
\end{equation}
is also quasi-isomorphic.
\begin{lemma}
There exists a subdivision $K^{(4)}$ of $K^{(3)}$ which satisfies the following:
There is a subcomplex $\widetilde{C}_0$ resp. 
$\widetilde{C}$ of $C_*(C_2^{(4)}, \QQ)$ resp. of $C_*({{\cal C}_1}^{(4)}, \QQ)$ which satisfies the following conditions.
\begin{enumerate}
\item The inclusion of $C_2^{(4)}$ into $\cal C_1^{(4)}$ induces  
an inclusion of $\widetilde{C}_0$ into $\widetilde{C}.$
\item The inclusion $i_0:\,\,\widetilde{C}_0\to C_*(C_2^{(4)}, \QQ)$
resp. $i_2:\,\,\widetilde{C}\to C_*({{\cal C}_1}^{(4)}, \QQ)$ is a quasi-isomorphism.

 \item For each chain $\ga\in \widetilde{C}$, $|\ga|$ and $|\delta \ga|$ meet all faces 
properly.
\end{enumerate}
\end{lemma} 
\begin{proof}

For each $j\geq 0$ take a basis 
of $H_j(C_2^{(3)},\QQ)$ as follows. First take a basis 
${\bf k}_j=\{k_{j,\al}\}_\al$ of the kernel of the map 
$(\iota_2)_*:\,H_j(C_2^{(3)},\QQ)\to H_j(C^{(3)}_1, \QQ)$ induced by the inclusion. 
Extend this to a basis ${\bf k}_j\cup {\bf l}_j=\{k_{j,\al}\}_\al\cup \{l_{j,\beta}\}_\beta$ of $H_j(C_2^{(3)},\QQ)$. Take a lift ${\bf K}_j\cup {\bf L}_j=\{K_{j,\al}\}_\al\cup \{L_{j,\beta}\}_\beta$
of ${\bf k}_j\cup {\bf l}_j$ in $C_j(C_2^{(3)},\QQ)$ and define $\widetilde{C_0}$ to be the complex generated by them for each $j$. 
Extend $\iota_2({\bf l}_j)$ to a basis $\iota_2({\bf l}_j)\cup {\bf m}_j
=\{\iota_2(l_{j,\beta})\}_\beta\cup \{m_{j,\xi}\}_\xi$ of $ H_j(C_1^{(3)}, 
\QQ)$.
Take a lift ${\bf M}_j=\{M_{j,\xi}\}_\xi$ of ${\bf m}_j$ in $ C_j(C_1^{(3)}, 
\QQ)$. 
Let ${\bf \Gamma}_j=\{\Gamma_{j,\al}\}_\al$ be chains in $ C_{j+1}(C_1^{(3)}, 
\QQ)$ such that $\delta \Gamma_{j,\al}=\iota_2(K_{j,\al})$ for each $\al$.
Set ${\cal X}=\{\underset{j,\al}\cup \iota_2(|K_{j,\al}|)\}\cup 
\{\underset{j, \beta}\cup \iota_2(| L_{j,\beta}|)\}\cup \{\underset{j,\xi}\cup |M_{j,\xi}|\}
$ and $X_0=\{\underset{j,\al}\cup \iota_2(|K_{j,\al}|)\}\cup 
\{\underset{j,\beta}\cup \iota_2(|L_{j,\beta}|)\}$. By the construction
$\cal X$ does not meet $|{\cal N}_1|$ so that  ${\cal X}-X_0\subset \overset{\circ}{|{\cal C}_1|}.$  Hence 
by Theorem \ref{Zeeman moving} 
$\cal X$ can be  moved in $\overset{\circ}{|{\cal C}_1|}$ by an isotopy which keeps $X_0$ fixed, so that $h_1({\cal X}-X_0)$ meet $F_2$ properly. Similarly the chains $\Gamma_{j,\al}$ can be moved keeping $\{\underset{j}\cup \iota_2(|{\bf K}_j|)\}\cup 
\{\underset{j}\cup \iota_1(|{\bf L}_j|)\}$ fixed, so that
$h_1(|\Gamma_{j,\al}|)$ meet $F_2$ properly.  Let $\widetilde{C}$ be the complex generated in degree $j$ by $\{h_1(\Gamma_{j-1,\al})\}_\al, 
\{\iota_2(K_{j,\al})\}_\al, \{\iota_2(L_{j,\beta})\}_\beta$ and $\{h_1(M_{j,\xi})\}_\xi$.
\end{proof}

We take a good ordering of $K^{(4)}$ with respect to $L_2^{(4)}.$ We have the following diagram.
\[
\begin{matrix}
C^*(K, N_1\cup N_2)&\xrightarrow[]{\iota_0}&\text{Cone}\,\bigl(C^*(K, N_1)&\xrightarrow{\text{res}}&C^*(N_2, N_1\cap N_2)\bigr)\\
\phantom{=}\downarrow{=}&&\phantom{\text{n}_2}\downarrow{\text{n}_2}
&&\phantom{\text{n}_1}\downarrow{\text{n}_1}\\
C^*(K, N_1\cup N_2)&\xrightarrow[]{\iota_1}&\text{Cone}\,\bigl(C^*(K,{{\cal N}_1})&\xrightarrow{i^*}&C^*(L_2,{{\cal N}_1}\cap L_2)\bigr)\\
\phantom{\cap \mu_{C_2}}\downarrow{\cap \mu_{C_2}}&&\phantom{\cap \mu_{{{\cal C}_1}}}\downarrow{\cap \mu_{{{\cal C}_1}}}&
H&\phantom{\cap \mu_{L_2}}\downarrow{\cap \mu_{L_2}}\\
C_*(C_2)&\xrightarrow[]{\iota_2}&\text{Cone}\,\bigl(C_*({\cal C}_1)&\xrightarrow{\cap T_{F_2}}& C_*(L_2\cap {{\cal C}_1})\bigr)\\
\phantom{i_0}\uparrow{i_0}&&\phantom{i_2}\uparrow{i_2}&&\uparrow{=}\\
\widetilde{C_0}&\xrightarrow{\iota_3}&\text{Cone}\,\bigl(\widetilde{C}\quad\quad\quad&\xrightarrow{\cap T_{F_2}}&C_*(L_2\cap  {{\cal C}_1})\bigr)\\
\phantom{i_0}\downarrow{i_0}&&\phantom{\text{incl}}\downarrow{\text{incl}}&&\phantom{\iota_5}\downarrow{\iota_5}\\
C_*(C_2)&\xrightarrow{\iota_4}&\text{Cone}\,\bigl(AC'(X,F_1)\quad&\xrightarrow{\cap T_{F_2}}&AC(F_2, F_1\cap F_2)\bigr)
\end{matrix}
\] 
Here we omitted the superscript ${}^{(4)}$ which should be put on every simplicial complex
appearing in the diagram, and the coefficient $\QQ$.
The chain $\mu_{C_2}$ resp. $\mu_{\cal C_1}$ resp. $\mu_{L_2}$ is the fundamental cycle of $C_2$ resp. $\cal C_1$ resp. $L_2$. By Lefschetz duality Theorem 3.43 \cite{Hat}, 
the maps
$\cap\mu_*$ for $*=C_2$, $\cal C_1$ and $L_2$ are quasi-isomorphic. The map res
is the restriction. The map $\text{n}_2$ is the inclusion, and $\text{n}_1$ is the composition of $i^*$ of Lemma \ref{n2l2} (2) and 
the map in (\ref{neighbor comparison}). 
The maps $\iota_1$ through $\iota_5$ are induced by  inclusions.  All the squares except for the one with an $H$ in it is commutative. The one
with $H$ is homotopy commutative. The complex $AC'(X,F_1)$ is defined as follows:
Let $AC(X,F_1)$ be the complex defined as in Definition \ref{doubleAC}
for the face $F_1=H_0\cup \cdots \cup H_k$. $AC'(X,F_1)$ is the subcomplex of
$AC(X,F_1)$ which consists of the chains $\ga$ such that $|\ga|$ and $|\delta\ga|$
meet $F_2$ properly. By Theorem \ref{Zeeman moving} $AC'(X,F_1)$ 
is quasi-isomorphic to $AC(X,F_1)$. The map $\iota_0$ is quasi-isomorphic. Since 
$\text{n}_1$ and $\text{n}_2$ are quasi-isomorphic, $\iota_1$ is quasi-isomorphic. 
The maps $\iota_2$ and $\iota_3$ are quasi-isomorphic by similar reasons.

\noindent We have the following commutative diagram.
\[
\begin{matrix}
C_*({{\cal C}_1})&\xrightarrow{i_1}&AC(X,F_1)\\
\phantom{i_2}\uparrow{i_2}&&\uparrow{i_3}\\
\widetilde{C}&\xrightarrow{\text{incl}}&AC'(X,F_1)
\end{matrix}
\]
Here the maps $I_1$ and $i_3$ are natural inclusions. The map $i_1$ is a quasi-isomorphism by the induction hypothesis, and the maps $i_2$ and $i_3$ are quasi-isomorphic by the
construction. Hence the map incl is also quasi-isomorphic. Since the map 
$\iota_5$ is quasi-isomorphic by the induction hypothesis,  the map $\iota_4$ is a quasi-isomorphism. The complex 
\[\text{Cone}\,\bigl(AC'(X,F_1)\xrightarrow{\cap T_{F_2}}AC(F_2, F_1\cap F_2)\bigr)\]
is $AC(X, \bf H)$ for $t=k+1$. 
Let $(K^{(2)})'$ be a derived of $K^{(2)}$ near $L_1\cup L_2^{(2)}$. Set $\ov{N_2}=N(L_1\cup L_2, (K^{(2)})')$
and $\ov{C_2}=C(L_1\cup L_2, (K^{(2)})')$.  Then by Lemma \ref{reflexivity} and the assumption that $L_1\cup L_2\lhd K$, 
$|\ov{C_2}|$ is a regular neighborhood of $|C_2|$. It follows that
the map of complexes
\[C_*(C_2^{(2)})\to C_*(\ov{C_2})\]
induced by the inclusion $C_2^{(2)}\to \ov{C_2}$ is a quasi-isomorphism.
By the uniqueness of regular neighborhood Theorem 3.24 \cite{RS}, we have the assertion. 
\end{proof}
\noindent We have the following commutative diagram.
\[
\begin{matrix}
A^*(X)(\log {\bf H})&\otimes& AC^{2m-*}(X,\bf H)&\xrightarrow{(\bullet, \bullet)}&\CC\\
\phantom{=}\uparrow{=}&&\uparrow&&\phantom{=}\uparrow{=}\\
A^*(X)(\log {\bf H})&\otimes&C_*(\ov{C_2})&\xrightarrow{(\bullet, \bullet)}&\CC.
\end{matrix}
\]
Here the map $(\bullet, \bullet)$ is the one defined in Notation \ref{res}. In this diagram all the vertical arrows are quasi-isomorphic, and the bottom line of this diagram induces the duality pairing
\[H^k_{dR}(X-{\bf H})\otimes H_k(|\ov{C_2}|;\CC)\to \CC\]
by de Rham Theorem.
Hence the pairing of the first line also induces  perfect pairings on cohomology groups.
\end{proof}
\section{The Abel-Jacobi map for higher Chow cycles}  In the following we consider 
the case where  $X=Y\times (\PP^1)^n$
with $Y$ a smooth projective complex variety of dimension $m$. Let $z_i$ $(i=1,\cdots,n)$ be the affine 
coordinates of $(\PP^1)^n$. The faces are intersections of subvarieties $\{z_i=0\}$ and $\{z_i=\infty\}$ for $i=1,\cdots, n$. 
The divisor at infinity $D$ equals $\overset{n}{\underset{i=1}\cup}\{z_i=1\}$. The complement
$(\PP^1)^n-D$ is denoted by $\square^n$, and the union of all faces of $\s$ is denoted
by $\part\s$.
For $1\leq i\leq n$ and $\beta\in \{0,\infty\}$, let $H_{i,\beta}$
be the face $\{z_i=\al\}$. We define the ordering of codimension one faces of $\square^n$ by
\[H_{1,0}<H_{1,\infty}<H_{2,\infty}<H_{2,0}<H_{3,0}<H_{3,\infty}<H_{4,\infty}<\cdots.\]
Here the numbering starts from 0. We denote the set of the indices
$\{(1,0),(1,\infty),\cdots, (n,\infty)\}$ of codimension one faces of $\s$ by $T$.
For a subset $I$ of $T$, the face $\underset{i\in I}\cap H_i$ is denoted by $H_I$.

\noindent The group $G_n= \{\pm 1\}^n\rtimes S_n$\index{$G_n$} acts naturally on  $\s$ as follows. The subgroup
$\{\pm 1\}^n$ acts by the inversion of the coordinates $z_i$, and  the symmetric group $S_n$ acts by
permutation of $z_i$'s. 
Let $\sign: G_n \to \{\pm 1\}$ be the character which sends 
$(\eps_1, \cdots, \eps_n; \sigma)$ to $\eps_1\cdot \cdots\cdot\eps_n\cdot\sign(\sigma)$. 
The idempotent $\Alt=\Alt_n:=({1}/{|G_n|})\sum_{g\in G_n} \sign(g) g $ \index{$\Alt$}
in the group ring $\QQ[G_n]$ is called the alternating 
projector. For a $\Q[G_n]$-module M, 
the submodule
$$
M^{\alt}=\{\al \in M\mid \Alt \al=\al\}=\Alt(M)
$$
\index{$M^{\alt}$}
is called the alternating part of $M$. The product $\s\times Y$ resp. $\part\s\times Y$
will be written $\s_Y$ resp. $\part\s_Y$ for short.
We recall the definition of the cubical version of higher Chow groups and of the Abel-Jacobi map.
We denote by $z^p(X,n)$ the free $\Q$-vector space
generated by subvarieties of $\s_Y$ meeting faces properly. The differential  is defined to be 
$$\part_\square=\sum_{j=1}^n(-1)^{j-1}(\part_0^j-\part_\infty^j):\,z^p(Y,n)^{\alt}\to z^p(Y,n-1)^{\alt}$$
Here the map $\part_i^0$ resp. $\part_i^\infty$ is the face map of the face $H_{i,0}$ resp. $H_{i,\infty}.$
By definition 
\[CH^p(Y,n)=H_n(z^p(Y,*)^{\alt}).\]

\noindent 
\begin{lemma}
\label{purity}
 Let $Z$ be an element of $z^p(Y,n)^{\alt}$.
\begin{enumerate}
\item $H^j_{|Z|}(\square^n_Y, \part \square^n_Y;\QQ(p))=0$ for $j<2p$.

\item There is an isomorphism
$$H^{2p}_{|Z|}(\square^n_Y, \part \square^n_Y;\QQ(p))= \underset{i}\cap \Ker( \part_i: \,\,
H^{2p}_{|Z|}(\s_Y,\QQ(p))\to H^{2p}_{|Z|\cap {H_i}_Y}({H_i}_Y,\QQ(p)))$$
Here the intersection on the right hand side is taken over all codimension one faces of $\s$.
\end{enumerate}
\end{lemma}

\begin{proof}
There is a spectral sequence
\[E^{ab}_1=\underset{\sharp I=a}\oplus H^b_{|Z|\cap H_I}(H_I;\QQ(p))\Rightarrow H^{a+b}_{|Z|}(\square^n_Y,\part \square^n_Y;\QQ(p)).\]
Since $|Z|$ meets every face properly, we have
$H^b_{|Z|\cap H_I}(H_I;\QQ(p))=0$  for  $j<2p$  and 
$$
\begin{array}{ll}
&H^{2p}_{|Z|}(\square^n_Y, \part \square^n_Y;\QQ(p))\\
=&E_2^{0\,2p}\\
=&\Ker(\underset{i}\oplus \part_i:\,H^{2p}_{|Z|}(\s_Y,\QQ(p))\to \underset{i}\oplus 
H^{2p}_{|Z|\cap {H_i}_Y}({H_i}_Y,\QQ(p))).
\end{array}
$$ The assertion follows from this.
\end{proof}   
There is a localization sequence of mixed Hodge structures
\[
\begin{array}{l}
0\to H^{2p-1}(\square^n_Y, \part \square^n_Y;\QQ(p))\overset{\beta}\to  H^{2p-1}(\square^n_Y-|Z|, \part \square^n_Y-|Z|;
\QQ(p))\\\to H^{2p}_{|Z|}(\square^n_Y, \part \square^n_Y; \QQ(p))
\overset{\iota}\to  H^{2p}(\square^n_Y, \part \square^n_Y; \QQ(p)).
\end{array}\]
As we will see shortly, there is an isomorphism 
$H^j((\square^n_Y, \part \square^n_Y; \QQ(p))\simeq H^{j-n}(Y,\QQ(p))$. The map $\beta$
induces an isomorphism
\begin{equation}
\label{Int Jac}
\begin{array}{rl}
J^{p,n}(Y):=&\displaystyle\frac{H^{2p-1-n}(Y,\CC)}{F^pH^{2p-1-n}(Y,\CC)+H^{2p-1-n}(Y,\QQ(p))}\\
\simeq &\displaystyle \frac{H^{2p-1}(\square^n_Y, \part \square^n_Y;\CC)}{F^pH^{2p-1}(\square^n_Y, \part \square^n_Y;\CC)
+H^{2p-1}(\s_Y,\part \s_Y;\QQ(p))}\\
\simeq& \displaystyle \frac{H^{2p-1}(\square^n_Y-|Z|, \part \square^n_Y-|Z|; \CC)}
{F^pH^{2p-1}(\square^n_Y-|Z|, \part \square^n_Y-|Z|; \CC)+H^{2p-1}(\s_Y, \part \s_Y;\QQ(p))}.
\end{array}
\end{equation}
 Suppose that $Z\in \Ker\part_\square\subset z^p(Y,n)^{\alt}.$ Since $Z\in z^p(Y,n)^{\alt}$, we have $\part_i Z=0$ for each codimension one face
$H_i$ and so we have a class $cl(Z)\in H^{2p}_{|Z|}(\square^n_Y, \part \square^n_Y;\QQ(p))$ .   If $Z$ is homologous to zero
i.e. $\iota(cl(Z))=0$,  there is a class $\widetilde{cl(Z)}\in H^{2p-1}(\square^n_Y-|Z|, \part \square^n_Y-|Z|;
\QQ(p))$ whose image under the boundary map is $cl(Z)$. 
We denote  by $z^p_{\hom}(Y,n)$ the subspace $\Ker \part_\square\cap \Ker \iota$
of $z^p(Y,n)^{\alt}$.
\begin{definition}
\label{Abel-Jacobi}
The map 
$$\Psi^{p,n}: z^p_{\hom}(Y,n)\to J^{p,n}(Y) ,\,\,Z\mapsto \rm{ the\,\, class\,\, of\,\, } \widetilde{cl(Z)}$$
is called the Abel-Jacobi map.
\end{definition}
\begin{remark} One can also define the Abel-Jacobi map in terms of Deligne-Beilinson cohomology.
The equality of the two definitions is proved in Theorem 7.11 \cite{E-V}.
\end{remark}
 
\subsection{An explicit description of the Abel-Jacobi map}
We will give a more explicit description of the map $\Psi^{p,n}$. Let $Z$ be an element of
$z^p_{\hom}(Y,n)$.  We replace the complex $AC^*(\square^n_Y,
\part\square^n_Y)$ with a smaller complex $AC^*_\square(\s_Y)$. In the following, the coefficients of the complex
$AC^*$ will be $\QQ$.
\begin{definition}
\label{cubical complex}
Let $AC_\square^*(\square^n_Y)$ be the simple complex associated to the double complex
\[AC_\square^{a,b}(
\square^n_Y)= AC_{2(m+n-a)-b}((\PP^1)^{n-a}_Y, D_Y)^{\alt} \quad (a\geq 0)\]
where the first  differential is $\part_\square$ and the second differential  is $\delta$. 
\end{definition}
 
\begin{proposition}

Let the map $\sum:AC^k(\s_Y,\part\s_Y)\to AC_\square^k(\s_Y)$ be the map
defined as 
$$\ga_{(\al_0<\cdots <\al_p)}\mapsto (-1)^{\overset{p}{\underset{j=0}\sum}\al_j}\ga.$$
Then the map $\Alt\circ \sum$ is a quasi-isomorphism.
\end{proposition}

\begin{proof}
We can argue as in \cite{Sch} 1.7. That the map $\Alt\circ \sum$ is a map of complexes can be checked directly.
The cohomology group $H^k(AC(\s_Y,\part\s_Y))$ can be computed in terms of the
spectral sequence
\[E_0^{ab}=\underset{I\subset T,\, \sharp I=a}\oplus
AC_{2(m+n-a)-b}((\PP^1)^{n-a}_Y,D_Y).\]
By the homotopy invariance one sees that 
$$E_1^{a b}=\underset{I\subset T,\, \sharp I=a}\oplus H^b(Y, \QQ)$$ 
and 
$$E_2^{ab}=H^b(Y,\QQ)\otimes H^a(\square^n, \part \square^n;\QQ)=
\begin{cases}
0& a\neq n\\
 H^b(Y,\QQ)& a=n.
\end{cases}
$$
On the other hand, the cohomology of $AC^*_\square(\s_Y)$
is computed by the spectral sequence
\[{}_\square E_0^{ab}=
AC_{2(m+n-a)-b}((\PP^1)^{n-a}_Y,D_Y)^{\alt}.\]
We see that ${}_\square E_1^{ab}=H^b(\square^{n-a} Y;\QQ)^{\alt}=0$ unless $a=n$. We see the map 
$\text{Alt}\circ \sum$ induces an isomorphism on the $E_2$-terms of the spectral sequences.
\end{proof}

\begin{proposition}
\label{alt}
Let $Z\in  z^p(Y,n)^{\alt}.$
\begin{enumerate}
\item We have
\[H^j(Cone(AC_\square^*(\square^n_Y)\to AC_\square^*(\square^n_Y-|Z|)))=0\]
for $j<2p$, and 
\[
\begin{array}{ll}
&H^{2p}(Cone(AC_\square^*(\square^n_Y)\to AC_\square^*(\square^n_Y-|Z|)))\\
\simeq & H^{2p}_{|Z|}( \square^n_Y, \part \square^n_Y;\QQ)^{\alt}
\end{array}\]

\item \[H^{2p-1}(AC_\square^*(\square^n_Y-|Z|))
\simeq H^{2p-1}(\square^n_Y-|Z|, \part \square^n_Y-|Z|;\QQ)^{\alt}.\]
\end{enumerate}
\end{proposition}
\begin{proof}
(1). There is a spectral sequence
\[
\begin{array}{l}
E_0^{ab}=AC_{2(m+n-a)-b}((\PP^1)^{n-a}_Y, D;\QQ)\oplus AC_{2(m+n-a)-b+1}((\PP^1)^{n-a}_Y,D\cup |Z|;\QQ)\\
\Rightarrow
H^{a+b}(Cone(AC_\square^*(\square^n_Y)\to AC_\square^*(\square^n_Y-|Z|))).
\end{array}\]
We have
\[
E_1^{ab}
=H^b_{|Z|\cap \square^{n-a} Y }(\square^{n-a} Y;\QQ)^{\alt}.
\]
By purity, $E_1^{ab}=0$ for $b<2p$. It follows that 
\[
\begin{array}{ll}
&H^{2p}(Cone(AC_\square^*(\square^n_Y)\to AC_\square^*(\square^n_Y-|Z|)))\\
\simeq &E_2^{0\,2p}=\Ker
\bigl(\part_\square: H^{2p}_{|Z|} (\square^n_Y;\QQ)^{\alt}\to  H^{2p}_{|Z|\cap\square^{n-1} Y} (\square^{n-1} Y;\QQ)^{\alt}\bigr).
\end{array}\]
The assertion (1) follows this and Lemma \ref{purity} (2).

(2). We have exact sequences
\[
\begin{array}{ll}
\text{(A)}:\,\,
&0\to H^{2p-1}(\square^n_Y,\part \square^n_Y)^{\alt}\to 
H^{2p-1}(\square^n_Y-|Z|,\part  \square^n_Y-|Z|)^{\alt}\\
&\to 
H^{2p}_{|Z|}(\square^n_Y,\part \square^n_Y)^{\alt}\to 
H^{2p}(\square^n_Y,\part \square^n_Y)^{\alt}
\end{array}
\]
\[
\begin{array}{ll}
\text{(B)}:&\,\,0\to H^{2p-1}(AC_\square^*(\square^n_Y))\to
H^{2p-1}(AC_\square^*(\square^n_Y-|Z|))\\
&\to H^{2p}(Cone(AC_\square^*(\square^n_Y)\to AC_\square^*(\square^n_Y-|Z|)))\\
&\to H^{2p}(AC_\square^*(\square^n_Y)).
\end{array}
\] We have a map of complexes $\text{(A)}\to \text{(B)}$ induced by the map ${\rm Alt}\circ \Sigma$. The assertion (2) follows from the five lemma. 
\end{proof}
As is defined in \cite{B3}, the cohomology class $cl(Z)\in H^{2p}_{|Z|}(\square^nY,\part \square^n_Y;\QQ)$  is represented by 
the cocycle 
\[(Z,0)\in Cone^{2p}\bigl(AC_\square^* (\square^nY)\to AC_\square^* (\square^nY-|Z|)
\bigr).\]

Since  $Z\in z^p_{\hom}(Y,n)$, there is an element $\Gamma\in AC_\square^{2p-1}(\s_Y)$ such that 
$d\Gamma=Z$.  By definition $\Psi^{p,n}(Z)$ is the class of $\Gamma$ in $J^{p,n}(Y)$. We will give
a description of $\Psi^{p,n}(Z)$ in terms of currents.  

\noindent We have
\begin{equation}
\label{cptfication}
\begin{array}{rl}
J^{p,n}(Y)=&\displaystyle\frac{H^{2p-1-n}(Y,\CC)}{F^pH^{2p-1-n}(Y,\CC)+H^{2p-1-n}(Y,\QQ(p))}\\
\simeq &\displaystyle \frac{H^{2p-1}(\square^n_Y, \part \square^n_Y;\CC)^{\alt}}{\bigl(F^pH^{2p-1}(\square^n_Y, \part \square^n_Y;\CC)
+H^{2p-1}(\s_Y,\part \s_Y;\QQ(p))\bigr)^{\alt}}\\
\simeq& \displaystyle \frac{H^{2p-1}((\PP^1)^n_Y, \part (\PP^1)^n_Y;\CC)^{\alt}}{\bigl(F^pH^{2p-1}((\PP^1)^n_Y, \part (\PP^1)^n_Y;\CC)
+H^{2p-1}(\PP^1_Y,\part (\PP^1)^n_Y;\QQ(p))\bigr)^{\alt}}\\
\simeq&\displaystyle\frac{\bigl(F^{m+n-p+1}H^{2(m+n-p)+1}(((\PP^1)^n- \part (\PP^1)^n)_Y,\CC)^{\alt}\bigr)^\vee}
{H_{2(m+n-p)+1}(((\PP^1)^n- \part (\PP^1)^n)_Y,\QQ(p))^{\alt}}.
\end{array}
\end{equation}
We denote by $\omega_n$ the meromorphic form $\dis \frac{1}{(2\pi i)^n}\frac{dz_1}{z_1}\we\cdots \we \frac{dz_n}{z_n}$ on $\s$.
 The map
\[A^*(Y)\to A^*(((\PP^1)^n- \part (\PP^1)^n)_Y)^{\alt}[n],\,\,\varphi\mapsto p_1^*\varphi\we \omega_n\]
is a canonical quasi-isomorphism. Here the map $p_1:\,\,((\PP^1)^n- \part (\PP^1)^n)_Y\to Y$ is the projection.

\begin{proposition}
\label{special C-S formula}
Let $\varphi$ be a smooth $p$-form on $Y$ 
and $\ga$ be an element of $AC_{p+1+n}(\square^n_Y)^{\alt}$. Then under the notation
of Notation \ref{res}, we have the equality
\[(-1)^{p}( \part_\square \ga, \varphi\we 
\omega_{n-1})=(\delta\ga, p_1^*\varphi\we \omega_n)-(\ga ,
p_1^*d\varphi\we\omega_n).\]
\end{proposition}

\begin{proof} The proof is essentially the same as that of Theorem 4.3 \cite{part II}, with the necessary
modification similar to the case of the proof of Proposition \ref{C-S}.
\end{proof}

\begin{corollary}
\label{proper duality}
Let $\e_3$ be a function on $AC^*_\square(Y)$ which takes the value $\frac{j(j+1)}2+a$ on $AC_j(\square^{n-a}_Y)^{\alt}.$
Then the pairing
\[(\bullet,\bullet)_\square,\,\,AC^*_\square (\square^n_Y)\otimes  A^*(Y)\to \CC[-2m-n]
 \]
which sends $\ga\otimes\varphi$ with  $\ga\in AC_{j+n-a}((\PP^1)^{n-a}_Y, D_Y)^{\alt}$ and $\varphi\in A^j(Y)$
to $(-1)^{\e_3(\ga)}\dis \int_\ga p_1^*\varphi\we \omega_{n-a}$
 is a map of complexes.
\end{corollary}

\begin{corollary}
\label{A-J}
Let $Z\in \Ker \part_\square \subset z^p(Y,n)^{\alt}_{hom}$ and $\Gamma\in AC^{2p-1}_\square(\square^n_Y)$
be an element  such that $d\Gamma=Z$.  
\begin{enumerate}
\item The map $\varphi\mapsto ( \Gamma,\varphi)_\square$
gives a well defined element of 
\newline$(F^{m-p+1}H^{2(m-p)+n+1}(Y,\CC))^\vee$.

\item  Modulo $H_{2(m-p)+n+1}(Y,\QQ(p))$
this element of (1) is independent of the choice of $\Gamma$. Hence this map defines a well defined element of $J^{p,n}(Y)$,
which is equal to $\Psi_{p,n}(Z)$.
\end{enumerate}
\end{corollary}

\begin{proof}
(1).  If a closed form $\varphi\in F^{m-p+1}A^{2(m-p)+n+1}(Y)$ is exact, $\varphi=d\psi$ for an element
$\psi\in F^{m-p+1}A^{2(m-p)+n}(Y)$ by Hodge theory. We have
\[(\varphi,\Gamma)_\square=\pm (\psi, d\Gamma)_\square=\pm (\psi, Z)_\square=0\]
for reason of type. 

(2). If $d\Gamma'=Z$, then $\Gamma-\Gamma'$ defines an element of 
$$
\begin{array}{ll}
&H^{2p-1}(\square^n_Y,\part\square^n_Y;\QQ(p))^{\alt}\\
\simeq &H^{2p-1}((\PP^1)^n_Y,\part(\PP^1)^n_Y;\QQ(p))^{\alt}\\
\simeq &H_{2(m+n-p)+1}(((\PP^1)^n-\part (\PP^1)^n)_Y,\QQ(p))^{\alt}\\
\simeq &H_{2(m-p)+n+1}(Y,\QQ(p)).
\end{array}
$$
\end{proof}
We will describe how the chain $\Gamma$ looks.  Since $H^j(\square^n_Y,\QQ)^{\alt}=0$ for $n>0$, there is an element
$\ga_n\in AC_{2(m+n-p)+1}(\square^n_Y)$ with $\delta\ga_n=Z$. Since $\delta\part_\square \ga_n=
\part_\square \delta\ga_n=\part_\square Z=0$, if $n-1>0$ there is an element $\ga_{n-1}
\in AC_{2(m+n-p)-1}(\square^{n-1}_Y)$ with $\delta\ga_{n-1}=\part_\square \ga_n$. One inductively finds
chains $\ga_j\in AC_{2(m-p)+n+j+1}(\square^j_Y)$, $j=1,\cdots, n$ such that 
$\part_\square \ga_{j+1}=\delta\ga_j$. The chain $\part\ga_1\in C_{2(m-p)+n}(Y)$ defines the homology class
of $Z$ in $H_{2(m-p)+n}(Y,\QQ)$ up to sign. Since $Z\in z^p(Y,n)^{\alt}_{hom}$, there is a chain
$\ga_0\in C_{2(m-p)+n+1}(Y)$ such that $\part_\square\ga_1=\delta\ga_0$. Then the chain $\Gamma$
equals the sum $\sum_{j=0}^n (-1)^{\frac{j(j-1)}2}\ga_{n-j}$.  Hence we have the following. 
\begin{corollary}
\[\Psi^{p,n}(Z)(\varphi)=\sum_{j=0}^n(-1)^{\e_3(\ga_{n-j})+\frac{j(j-1)}2}\int_{\ga_{n-j}} p_1^*\varphi\we \omega_{n-j}.\]
\end{corollary}

Next we consider the Abel-Jacobi map for open varieties, and of the relative higher Chow cycles.
Let $Y$ be a smooth projective variety, and suppose that ${\bf D}=D_0\cup D_1\cup\cdots \cup D_s$
and $V=V_0\cup \cdots \cup V_t$ are strict normal crossing divisors such that ${\bf D}\cup V$
is also a strict normal crossing divisor. We denote $U=Y-{\bf D}$, $S=\{0,\cdots,s\}$ and $T=
\{0,\cdots, t\}$. The subvariety of $Y$ of the form $V_I:=\underset{i\in I}\cap V_i$ for a subset $I$ of $T$,
is called a $V$-face. A face of $\square^l$ is called a {\it cubical face}.  The variety $\square^l$ itself is
a cubical face. For a subset $I$ of $T$ and a subset $J$ of $S$,  
a face of $(D_J\cap V_I)\times \square^l$ is the product of a $V$-face and a cubical face. Let $z^r(D_J\cap V_I,l)^{\alt}_V$ be the subspace
of $z^r(D_J\cap V_I,l)^{\alt}$ which consists of the cycles $z$ which meet all the {\it faces}
of $(D_J\cap V_I)\times \square^l$ properly.  By the moving Lemma (\cite{B1}, \cite{Lv}), the complex
$z^r(D_J\cap V_I,*)^{\alt}_V$ with the differential $\part_\square$ is quasi-isomorphic to $z^r(D_J\cap V_I,*)^{\alt}.$
Let $z^r(U,V,*)$ be the complex defined by the equality 
\[z^r(U,V,n)=\underset{
\substack{
I\subset T,\,J\subset S\\
\sharp J-\sharp I+l=n
}}\oplus 
z^{r-\sharp J}(D_J\cap V_I,l)_V^{\alt}
\]
where the differential is defined by
\[d=\part_D+(-1)^{\sharp J}\part_V+(-1)^{\sharp J+\sharp I}\part_\square\]
on $z^{r-\sharp J}(D_J\cap V_I,m)_V$. Then the higher Chow group of $U$ relative to $V$ is defined  by the equality
\[CH^r(U,V;n)=H_n(z^r(U,V,*)).\]  
Accordingly we consider the following complex of admissible chains.
\[AC^j_\square(\square^n_U,\square^n_V;n)=\underset{\substack{I\subset T,\,J\subset S,\,l\leq n\\
j=2m+n-\sharp J+\sharp I+l-i}}\oplus AC_i((D_J\cap V_I)\times (\PP^1)^l, (D_J\cap V_I)\times D)_V^{\alt}.
\]
Here $AC_i((D_J\cap V_I)\times (\PP^1)^l, (D_J\cap V_I)\times D)_V^{\alt}$ is the subspace of $AC_i((D_J\cap V_I)\times (\PP^1)^l, (D_J\cap V_I)\times D)^{\alt}$ 
which consists of the chains $\ga$ such that $|\ga|$ and $|\delta\ga|$ meet all the {\it faces} of 
$(D_J\cap V_I)\times \square^l$ properly. 
The differential is defined by
\[d=\part_D+(-1)^{\sharp J}\part_V+(-1)^{\sharp J+\sharp I}\part_\square+(-1)^{\sharp J+\sharp I+l}\delta.\]
Then we have
\[H^j(AC^*_\square(\square^n_U,\square^n_V))=H^j( \square^n_U, ( \square^n_V)\cup (\part \square^n_U);\QQ^\alt).\]
In this case we also have a description of the Abel-Jacobi map in terms of
the admissible chains. Probably it is more adequate to give an example than writing
down the general formula. In the next subsection we will describe the case of polylog cycles rather explicitly.

\subsection{The case of polylogarithms}
We will see that the Hodge realization of polylog cycle constructed in \cite{BK} can be 
described in terms of the Abel-Jacobi map. Fix an integer $p\geq 1$, and for $1\leq k\leq p$ let
$Y_k=(\PP^1)^{k-1}$ with the affine coordinates $(x_1,\cdots, x_{k-1})$. Let $a \in \CC-\{0,1\}$.
For $1\leq j\leq k-1$ and $u\in \{a,a^{-1}\}$ let $D_{j,u}$ be the divisor on $Y_k$ defined by
$x_j=u$. Then the closed subset 
$${\bf D}_k:=\underset{\substack{1\leq j\leq k-1\\
u\in \{a,a^{-1}\}}}\cup D_{j,u}$$ 
is a normal crossing divisor of $Y_k$. 
We write the set of indices of the irreducible components of ${\bf D}_p$
\[\{(j,u)|\,1\leq j\leq p-1,\,u\in \{a,a^{-1}\}\}\]
by $S$, and define an ordering on $S$  by
\[(1,a)<(1,a^{-1})<(2,a^{-1})<(2,a)<(3,a)<(3,a^{-1})<(4,a^{-1})\cdots\]
where the numbering starts from zero.
We denote $Y_k-{\bf D}_k$ by $U_k$. We have 
\[CH^p(U_p, p)=H_p(Z(*))\]
where 
\[Z(n)=\underset{\substack{J\subset S,\,l\geq 0\\\sharp J+l=n}} \bigoplus z^{p-\sharp J}(D_J,l)^{\alt}.\]
with the differential $d=\part_D+(-1)^{\sharp J}\part_\square$ on $z^k(D_J,l)$. 
\noindent The group $G_{k-1}$ acts on $Y_k$ in a similar way as $G_l$ acts on $\square^l$, so there is an action
of the product $G_{k-1}\times G_l$ on $Y_k\times \square^l$. The character sign naturally extends to the product character
$\text{sign}_\times$ on $G_{k-1}\times G_l$, and let $\text{Alt}_\times$ be the associated projector. 
For a $\Q[(G_{k-1}\times G_l)]$-module $M$, we denote by $M^{\alt_\times}$ the module $\text{Alt}_\times (M).$
 Let ${\cal Z}(*)$ be the complex defined by the equality
\[{\cal Z}(n)=\bigoplus_{0\leq c\leq n}z^{p-c}(Y_{p-c},\,n-c)^{\alt_\times}.\]
The differential of ${\cal Z}(*)$ is defined as follows. For $u\in \{a,a^{-1}\}$ and $j\geq 1$, let $\iota_u^j:\,\,
Y_k\to Y_{k+1}$ be the inclusion defined by $(x_1,\cdots, x_{k-1})\mapsto (x_1,\cdots, x_{j-1},u,x_{j+1},\cdots, x_{k-1}).$
Let 
$$\part_Y: \,z^\bullet(Y_k, l)\to z^{\bullet+1}(Y_{k+1},l)$$ 
be the map $\sum_{j=1}^{k-1}(-1)^{j-1}((\iota_a^j)_*-(\iota_{a^{-1}}^j)_*)$.
The map $\part_Y$ is a differential i.e. $\part_Y^2=0$, and it induces a map on the alternating part: 
we have $\part_Y(z^\bullet(Y_k, l)^{\alt_\times})\subset z^{\bullet+1}(Y_{k+1},l)^{\alt_\times}$. 
 Let  
$$s:\,{\cal Z}(*)\to Z(*)$$ 
be the homomorphism defined as follows.
For a subset  $J=\{\beta_0<\cdots <\beta_{c-1}\}$ of $S$,  let $D_J$ be the corresponding face of $Y_p$. We identify $D_J$ with 
$Y_{p-c}=(\P^1)^{p-1-c}$ in the natural way. Then 
on $z^{p-c}(Y_{p-c},l)^\alt$, we define 
\begin{equation}
\label{maps}
s(z)=\sum_{J=\{\beta_0<\cdots <\beta_{c-1}\}}\bigl((-1)^{\sum_{j=1}^{c-1}\beta_j}z \in z^{p-c}(D_J, n-c)^\alt\bigr).
\end{equation}
One sees that the map $s$ is a homomorphism of complexes. 
  
Let  $\rho_k(a)$ be an element of $z^k(Y_k,k)^{\alt_\times}$ given parametrically
\[(-1)^{p-k}\{(x_1,\cdots,x_{k-1}, 1-x_1,1-\frac{x_2}{x_1},\cdots, 1-\frac{x_{k-1}}{x_{k-2}},
1-\frac{a}{x_{k-1}})\}^{\alt_\times}.\]
we have the equality $\part_\square \rho_{p-c}(a)=\part_Y\rho_{p-c-1}(a)$ for $0\leq c\leq p-2$. 
Let $R_p(a)\in {\cal Z(p)}$ be the element $\sum_{c=0}^{p-1}(-1)^{\frac{c(c+1)}2}\rho_{p-c}(a)$.
Then we have $dR_p(a)=0$.

For $k\geq 2$ or $i > 0$ we have $H^j(Y_k\times \square^i,\QQ)^{\alt_\times}=0.$ So we can find
  chains
\[\eta_{p-c}(i)\in AC_{p-c+i}(Y_{p-c}\times (\P^1)^{i+1},\,Y_{p-c}\times D)^{\alt_\times}\]
for $0\leq c\leq p-1$ and $0\leq i\leq p-c-1$ such that 
\[
\begin{array}{rl}
\delta \eta_{p-c}(p-c-1)=&\rho_{p-c}(a)\\
\delta\eta_{p-c}(i)=&\part_Y\eta_{p-c-1}(i)+(-1)^{p-c+i+1}\part_\square\eta_{p-c}(i+1),\,\,0\leq i\leq p-c-2.
\end{array}
\]
Explicitly,
\[
\begin{array}{ll}
&\eta_{p-c}(i)\\
=&\{(x_1,\cdots,x_i,t_{i+1},\cdots,t_{p-c-1},1-x_1,\cdots, 1-\frac{x_i}{x_{i-1}},1-\frac{t_i}{x_i})|\\
&t_{p-c-1}\in [0,a],\cdots,t_{k-2}\in [0,t_{k-1}],\cdots, t_i\in [0,t_{i+1}],\\
&x_1,\cdots,x_i\in \PP^1\}^{\alt_\times}.
\end{array}
\]
with suitable orientation. 
A complex computing the cohomology group $H^\bullet(\square^p_{U_p}, \part \square^p_{U_p};\QQ)$ is given by
\[AC^j_\square(\square^p_{U_p};p):=\underset{\substack{J\subset S,\,l\leq p\\
j=2(p-1)+p-\sharp J+l-i}}\oplus AC_i(D_J\times (\P^1)^l, D_J\times D)^{\alt}.
\]
where the differential is given by $\part_D+(-1)^{\sharp J}\part_\square+(-1)^{\sharp J
+l}\delta$ on $AC_i(D_J\times (\P^1)^l, D_J\times D)^{\alt}.$
Let ${\cal AC}^\bullet$ be the complex defined by
\[{\cal AC}^j=\bigoplus_{\substack{0\leq c\leq p-1,\,l\leq p\\ j=2(p-1)+p-c+l-i}}
AC_i(Y_{p-c}\times (\P^1)^l, Y_{p-c}\times D)^{\alt_\times}\]
with the differential given by $\part_Y+(-1)^c\part_\square+(-1)^{c+l}\delta$ on
 $AC_i(Y_{p-c}\times (\P^1)^l, Y_{p-c}\times D)^{\alt_\times}.$ The map
$s:\, {\cal AC}^\bullet \to AC^\bullet(\square^p_{U_p};p)$ defined in (\ref{maps}) is a homomorphism of
complexes. The element $R_p(a)$ can be regarded as a cocycle of ${\cal AC}^{2p}.$
Let ${\cal Y}_p(a)$ be the element of ${\cal AC}^{2p-1}$ defined as the sum
\[\sum_{0\leq c\leq p-1,\,0\leq i \leq p-c-1} (-1)^{\epsilon(c,i)}\eta_{p-c}(i)\]
where $\epsilon(c,i):=1+i(c+1+p)+\frac{c(c-1)}2$. 
Then we have
\[d{\cal Y}_p(a)=R_p(a).\]
Let $V$ be the divisor 
\[\overset{p-1}{\underset{i=1}\cup}\bigl(\{x_i=0\}\cup\{x_i=\infty\})\]
of $Y_p$. For each irreducible component $V_i$ of $V$, the pull-back of $s(\rho_{p-c}(a))$ and of
$s(\eta_{p-c}(i))$ to $V_i$ is zero for each $0\leq c\leq p-1$ and $0\leq i\leq p-c-1$. Hence
$s(R_p(a))$ defines an element of the relative higher Chow group $CH^p(U_p,V;p)$ and we have 
\[ds({\cal Y}_p(a))=s(R_p(a))\]
in $AC^{2p}_\square(\square^p_{U_p}, \square^p_V;p)$.
We have 
$$H^{2p-1}(\square^p_{U_p}, (\square^p_V)\cup( \part\square^p_{U_p});\QQ)^{\alt}
\simeq H^{p-1}(U_p,V;\QQ).$$
The dual space of  $H^{p-1}(U_p,V;\QQ)\otimes \CC$ is $H^{p-1}((\PP^1)^{p-1}-V, {\bf D}_p;\CC)$.
The form $\omega_{p-\sharp J-1}$ on $D_J$ defines an element of $H^{p-1}((\PP^1)^{p-1}-V; {\bf D}_p;\CC)$.
We have
\[(s(\cal Y_p(a)), (\omega_{p-\sharp J-1}\text{ on }D_J))=\sum_{i=0}^{p-\sharp J-1}(-1)^{\e(\sharp J,i)}
\int_{\eta_{p-\sharp J}(i)}\omega_{p-\sharp J-1}\we \omega_{i+1}.\]
When $i>0$, the integral $\int_{\eta_{p-\sharp J}(i)}\omega_{p-\sharp J-1}\we \omega_{i+1}=0$ for reason of type,
and
\[\begin{array}{ll}
&\dis\int_{\eta_{p-\sharp J}(0)}\omega_{p-\sharp J-1}\we \omega_1\\
=&\dis\frac{1}{(2\pi i)^{p-\sharp J}}
\int_0^a\frac{dt_{p-\sharp J-1}}{t_{p-\sharp J-1}}\cdots \int_{0}^{t_2}\frac{dt_1}{t_1}\int_0^{t_1}\frac{-dt_0}{1-t_0}\\
=&\dis\frac{1}{(2\pi i)^{p-\sharp J}}Li_{p-\sharp J}(a).
\end{array}\]

\end{document}